\documentclass[12pt]{amsart}
\usepackage{amssymb,latexsym, amscd,pb-diagram}
\usepackage{sseq}
\usepackage[all]{xy}
\usepackage[square, numbers]{natbib}
\usepackage{graphicx}
\usepackage{mathrsfs}
\usepackage{color}
\usepackage{amsmath}
\usepackage{multirow}
\usepackage{hyperref}
\usepackage{todonotes}

\usepackage[a4paper, total={6in, 8in}]{geometry}
\usepackage[utf8]{inputenc}
\usepackage[T1]{fontenc}
\usepackage{amsfonts,  amsthm}
\usepackage{mathtools}
\usepackage{txfonts}
\usepackage{mathrsfs}
\usepackage{ textcomp }

 \newtheorem{theorem}{Theorem}[section]
 \newtheorem{corollary}[theorem]{Corollary}
  
 \newtheorem{lemma}[theorem]{Lemma}
 \newtheorem{proposition}[theorem]{Proposition}
\newtheorem*{theorem*}{Theorem}
 \theoremstyle{definition}
 \newtheorem{definition}[theorem]{Definition}
\theoremstyle{definition}
\newtheorem*{definition*}{Definition}
 \theoremstyle{definition}
 \newtheorem{example}[theorem]{Example}

 \newtheorem{remark}[theorem]{Remark}
 
 \numberwithin{equation}{section}

\newcounter{commentcounter}

\begin{document}

\title[Magnetic Equivariant Graded Brauer Group]{Magnetic Equivariant Graded Brauer Group}



\author{Higinio Serrano}
\address{
Institute of Mathematics and Informatics, Bulgarian 
Academy of Sciences; Akad. G. Bonchev St, Bl. 8, 1113 Sofia, Bulgaria.}
\email{hserrano@math.cinvestav.mx}

\author{Bernardo Uribe}
\address{Departamento de Matem\'{a}ticas y Estad\'istica, Universidad del Norte, Km.5 V\'ia Antigua a Puerto Colombia, 
Barranquilla, 081007, Colombia.}
\email{bjongbloed@uninorte.edu.co, buribe@gmail.com}


\date{\today}
\begin{abstract}
Given a magnetic finite group, we consider the similarity classes of magnetic equivariant central simple graded algebras over the complex numbers. We call this set the magnetic equivariant graded Brauer group and its structure as an abelian group is explicitly determined. Following Karoubi, we
argue that the elements of this graded Brauer group parametrize the twistings of the magnetic equivariant K-theory of a point.
\end{abstract}

\maketitle

\tableofcontents

\section*{Introduction} 

Magnetic representations of groups are complex representations
where the elements of the group act either complex linearly or complex antilinearly. By fixing the elements of the group that act complex antilinearly, we obtain what is known as a magnetic group. This is a group $G$ together 
with a surjective homomorphism $\phi:G \to \mathbb{Z}/2$. This homomorphism aims to encode which elements of the group act complex linearly and which ones would act complex antilinearly; for the first set we take the elements of the subgroup $\phi^{-1}(0)$, while for the second set we take its complement $\phi^{-1}(1)$. The study of the representations of magnetic groups was done by Wigner
in his seminal book \cite{wigner}, and its classification on the three types ($\mathbb{R}$, $\mathbb{C}$ or $\mathbb{H}$) of irreducible representations was understood through the structure of the algebra of equivariant endomorphisms of the irreducible representations.

The generalization of this classification to magnetic equivariant $\mathbb{Z}/2$-graded representations was done in the 60's by Dyson \cite{Dyson_Threefold} and a ten-fold classification
was thus obtained. Around the same time, the generalization of the Brauer group of Morita equivalence classes of central simple algebras over a fixed field to the case of central simple graded algebras was carried out by Wall \cite{Wall_Graded_Brauer_groups}, and the graded Brauer group of a field was born. The graded Brauer group of the reals was shown to be a cyclic group of order eight, and the one of the complex numbers of order two. 

It was clear at the time that the Wedderburn-Artin theorem, on its graded version \cite[Thm. 6.2.5]{Varadarajan_Supersymmetry}, was the key ingredient to relate both approaches. On the one hand, the algebra of equivariant endomorphisms of an irreducible magnetic graded representation has the structure of a graded division algebra \cite{Dyson_Threefold}, while on the other, by the Wedderburn-Artin theorem, central simple graded algebras 
are isomorphic to graded matrix algebras over graded division algebras \cite{FreedMoore_Twisted_equivariant_matter}. 

In this work we introduce the appropriate definition of graded  central simple  algebras in the magnetic equivariant setup, we generalize the notion of similarity in this context, and we introduce their similarity classes with the name of {\it magnetic equivariant graded Brauer group}.

The equivariant Brauer group of a field was introduced by Fr\"ohlich \cite{Froelich} and its graded extension was introduced by Riehm \cite{Riehm_Graded_Equivariant_CSA}. Here we study central simple graded algebras over $\mathbb{C}$ endowed with a magnetic $G$-action of algebra automorphisms; we call these objects magnetic equivariant central simple graded algebras. Two of these algebras are called similar if after tensoring with the endomorphism algebra of degree zero magnetic representations they become $G$-equivariant isomorphic algebras. The similarity classes of these algebras is a finite group by the tensor product structure and it is our desired magnetic equivariant graded Brauer group.

The main result of this work is Thm. \ref{Theorem decomposition GrBr} where we present an explicit calculation of the abelian group structure of the magnetic equivariant graded Brauer group. This is the following isomorphism:
\begin{align} \nonumber
\mathrm{GrBr}_{(G,\phi)}(\mathbb{C}) \   {\cong}  \Big( H^2(G, \mathbb{C}^*_\phi) \overset{\bullet}{\times}_\rho \mathrm{Hom}(G,\mathbb{Z}/2) \Big) \overset{\bullet}\times_\beta \mathbb{Z}/2  
\end{align}
where the symbols $\overset{\bullet}{\times}_\rho$ 
and $\overset{\bullet}{\times}_\beta$ denote twisted products by explicit 2-cocycles $\rho$ and $\beta$.

We compare our calculation with the previously known calculations of  the graded equivariant Brauer groups of $\mathbb{R}$ and $\mathbb{C}$, and we highlight some of their similarities and differences. We present explicit calculations of the magnetic equivariant graded Brauer groups noting the 
key differences whenever the homomorphism $\phi: G \to \mathbb{Z}/2$ does not split.

We finish our work with an application of the magnetic equivariant graded Brauer groups in the classification of twists in magnetic equivariant K-theory \cite{serrano2025magneticequivariantktheory}. We follow Karoubi's approach to study K-theory, and we argue that the magnetic equivariant graded Brauer group parametrizes the twists in the magnetic equivariant K-theory.
We show that the degree shift isomorphism in magnetic equivariant K-theory follows from a property of the magnetic equivariant graded Brauer group, and we further
show that the 4-periodicity in certain  magnetic equivariant K-theories follow 
from the particular structure of the associated magnetic equivariant graded Brauer groups.

We include in the first section Table \ref{Table-tenfold-way} where we present a summary of the results in the theory of graded Brauer groups over $\mathbb{R}$ and $\mathbb{C}$, in the classification of real division graded algebras, and in the appearance of the ten-fold way in the study of symmetries of periodic Hamiltonians. The explicit algebras defined in the first section are the main characters appearing in the latter sections.










\section{Graded Brauer groups and the ten-fold way}

In this section we are going to explore the relation between the graded Brauer groups of the reals and the complex numbers, the set of real graded central division algebras, and the ten-fold classification of non-interacting fermionic Hamiltonians.

The famous Wedderburn-Artin theorem, in its graded version \cite[Thm. 6.2.5]{Varadarajan_Supersymmetry},
implies that any central simple graded algebra over a fixed field is a matrix
algebra over a graded division algebra over the field. 
Therefore, the classification of central simple graded algebras can be reduced to  the classification of graded division algebras. 

Two central simple graded algebras are called similar if they are matrix algebras over the same graded division algebra, and the similarity classes of simple central graded algebras is called the graded Brauer group. It becomes an abelian group because the tensor product of simple central graded algebras is also a central simple graded algebra.

In the case of the reals and the complex numbers the graded Brauer groups are respectively $\mathbb{Z}/8$ and $\mathbb{Z}/2$. The real and complex Clifford algebras can be taken as representatives of the similarity classes and the periodicity of the Clifford algebras is the famous Bott periodicity.

The ten possibilities of central simple graded algebras have also appeared in the physics literature in different disguises. Perhaps the first appearance was Dyson classification of symmetries of Quantum systems \cite{Dyson_Threefold}, while a more recent one is the famous Altland-Zirnabauer (AZ) classification of free-fermion periodic Hamiltonians \cite{Altland}. This classification is also known as the "Ten-fold way" and has been masterfully reinterpreted by Kitaev in terms of the classification of Clifford algebras \cite{kitaev}. 

Here we will summarize the classifications mentioned above and we will present a simple description that allows us to relate the graded Brauer groups over $\mathbb{R}$ and $\mathbb{C}$, the graded division algebras over $\mathbb{R}$, and the AZ classification scheme.
The intermediary group that will allow us to relate all these objects is the graded Brauer group of simple central graded complex algebras endowed with a $\mathbb{C}$-antilinear automorphism. The main results of this section are summarized in Table \ref{Table-tenfold-way}.

\subsection{Graded Brauer groups over $\mathbb{R}$ and $\mathbb{C}$}

Graded magnetic representations and their endomorphisms incorporate information on graded division algebras over both the field of complex numbers and the field of real numbers. To study both in the same framework we are going to take the algebraic approach of working over a
field $k$ taking into account its separable closure ${k}_s$ and 
the absolute Galois group $\mathrm{Gal}({k}_s/k)$. Our main application will
be whenever $k=\mathbb{R}$, ${k}_s=\mathbb{C}$ and the Galois group is 
$\mathrm{Gal}(\mathbb{C}/\mathbb{R}) \cong \mathbb{Z}/2$ acting on $\mathbb{C}$ by conjugation.

The main references for what follows are Wall's paper on
graded Bauer groups \cite{Wall_Graded_Brauer_groups},
its exposition on Lam's book \cite[\S 4]{Lam_Introduction_to_quadratic_forms}, the generalization for field extensions done in Gille and Szamuely's book \cite[\S 2 \& \S 4]{Gille_Central_Simple_algebras} and the equivariant
extension of Riehm \cite{Riehm_Graded_Equivariant_CSA}.


Assume the field $k$ if of characteristic different from $2$. All
algebras we will consider will be finite dimensional associative $k$-algebras with unit $1$, and  "super algebra", or simply "graded algebra",  will mean $\mathbb{Z}/2$-graded algebra. 

Thus a graded algebra $A$ is a direct sum $A = A_0 \oplus A_1$ with $A_iA_j \subset A_{i+j}$ for $i, j \in \mathbb{Z}/2$. The homogeneous elements of $A$ are $hA = A_0 \cup A_1$ and the degree function
is $\partial : hA \to \{0,1\}$. 

Graded operations will be denoted with a hat $\hat{\ }$, while the operations disregarding the grading will not contain a hat.

If $B$ is a subalgebra of $A$, its {\it centralizer} $C_A(B)$ consists of the elements in $A$ which commute with $B$:
\begin{align}
C_A(B) ) := \{ a \in A \colon  ab=ba \ \mathrm{for} \ \mathrm{all} \ b \in B \}
\end{align}
while the {\it graded centralizer} $\hat{C}_A(B)$ consists of all elements in $A$ which commute in the graded sense with all elements in $B$.
Recall that two homogeneous elements $a, b \in hA$ commute in the graded sense if $ab=(-1)^{\partial a\partial b}ba$.

The {\it center} $Z(A)$ of $A$ is the centralizer $C_A(A)$ and the {\it graded center} $\hat{Z}(A)$ of $A$ is the graded centralizer $\hat{C}_A(A)$. 
If $A_1=0$ then the two concepts agree.

$A$ is denoted {\it graded simple} if the only two-sided graded ideals are $0$ and $A$, and it is {\it central} if its graded center $\hat{Z}(A)$ is $k$.

Graded algebras, simple graded algebras, central graded algebras and central simple graded algebras over the field $k$ will be denoted respectively GA$(k)$, SGA$(k)$, CGA$(k)$ and CSGA$(k)$. 

The graded tensor product $A \hat{\otimes} B$ of two GAs is 
the usual tensor product as vector spaces, and the graded multiplication defined on
homogeneous elements $a,b,a',b'$ by the equation
\begin{align}
(a \otimes b)(a' \otimes b')=(-1)^{\partial a' \partial b} (aa' \otimes bb') 
\end{align}
makes $A \hat{\otimes} B$ into a GA. Moreover, if $A$ and $B$
are CSGA, then so is $A \hat{\otimes} B$.

If $V=V_0 \oplus V_1$ is a finite dimensional graded vector space over $k$, we let the algebra of endomorphisms $\mathrm{End}(V)$ of $V$
be graded algebra with homogeneous elements of degree zero the 
endomorphisms $\rho$ with $ \rho(V_i) \subset V_i $ and the homogeneous elements of degree one the ones which satisfy $ \rho(V_i) \subset V_{i+1}$. We will denote this CSGA by $\mathrm{End}(V_0,V_1)$. If $V$ is the graded
vector space with $V_0=k^p$ and $V_1=k^q$, we will use the matrix notation:
\begin{align}
    M_{p|q}(k) : = \mathrm{End}(k^p, k^q).
\end{align}
The matrix algebra $M_p(k)$ will simply denote $M_{p|0}(k).$

Two CSGAs  $A$ and $B$ are similar, denoted as $A \sim B$, if there are
graded vector spaces $V$ and $W$ such that there is an isomorphism:
\begin{align}
    A \hat{\otimes}\mathrm{End}(V_0,V_1) \cong B \hat{\otimes}\mathrm{End}(W_0,W_1)
\end{align}
as GAs. Similarity is an equivalence relation, and we denote by
$[A]$ the similarity class of the CSGA $A$. The equivalence classes of similar CSGAs is called the {\it Graded Brauer group}
or the {\it Brauer-Wall group}:
\begin{align}
    \mathrm{GrBr}(k) := \mathrm{CSGA}(k)/\sim.
\end{align}

The graded tensor product makes $\mathrm{GrBr(k)}$ into a monoid, and the opposite graded algebra:
\begin{align}
    A^{\mathrm{op}}= \{a^{\mathrm{op}} \colon a \in A \}, \ \ a^{\mathrm{op}}b^{\mathrm{op}}:=(-1)^{\partial a \partial b}(ab)^{\mathrm{op}}
\end{align}
provides the inverse. This follows from the isomorphism of graded algebras $A \hat{\otimes}A^{\mathrm{op}} \cong \mathrm{End}(A)$. Therefore
the Graded Brauer group of $k$ becomes an abelian group,
since we know that $A \hat{\otimes} B \cong B \hat{\otimes} A$ with $a \otimes b \mapsto (-1)^{\partial a \partial b}b \otimes a$.

The cases of interest for us are when the field is $\mathbb{R}$ or $\mathbb{C}$. In these cases we have:
\begin{align}
     \mathrm{GrBr}(\mathbb{R}) \cong \mathbb{Z}/8  \ \ \ \&  \ \ \ 
      \mathrm{GrBr}(\mathbb{C}) \cong \mathbb{Z}/2,
\end{align}
and the Clifford algebras provide representatives for the similarity classes of CSGA. For $p,q$ non-negative integers, the Clifford algebra $C^{p,q}_\mathbb{R}$ is the Clifford algebra of the vector space $\mathbb{R}^{p+q}$
with quadratic form $Q(x)=x_1^2+...+x_p^2-x_{p+1}^2-...-x_{p+q}^2$. Alternatively it could be written as:
\begin{align}
    C^{p,q}_\mathbb{R} = \mathbb{R}\langle e_1,..,e_{p+q} | e_ie_j=-e_je_i, e_j^2=Q(e_j) \rangle
\end{align}
where the generators $e_j$ have degree one. The Clifford algebras are CSGA
over $\mathbb{R}$ and satisfy several properties:
\begin{align}
C^{p,q}_\mathbb{R} \hat{\otimes} C^{r,s}_\mathbb{R} &\cong C^{p+r,q+s}_\mathbb{R} \ & \ C^{1,1}_\mathbb{R} &\cong \mathrm{End}(\mathbb{R}, \mathbb{R})\\ C^{p,q+8}_\mathbb{R} &\cong C^{p+8,q}_\mathbb{R} \ & \ C^{0,8}_\mathbb{R} &\cong \mathrm{End}(\mathbb{R}^{16}, \mathbb{R}^{16}).
\end{align}
Wall \cite{Wall_Graded_Brauer_groups} showed that the similarity class of the Clifford algebra $C^{0,1}_\mathbb{R}$ can be taken as a generator of the graded Brauer group of $\mathbb{R}$. That is:
\begin{align}
    \mathrm{GrBr}(\mathbb{R}) = \langle [C^{0,1}_\mathbb{R}] \rangle \cong \mathbb{Z}/8.
\end{align}

For the complex numbers there are only two similarity classes and we have the isomorphisms:
\begin{align}
    \mathrm{GrBr}(\mathbb{C}) = \langle [C^{1}_\mathbb{C}] \rangle \cong \mathbb{Z}/2.
\end{align}
Here $C^{n}_\mathbb{C}$ denotes the complex Clifford algebra with $n$-generators and quadratic form  $Q(x)=|x|^2$. Note that for the complex numbers all quadratic forms are equivalent and therefore $C^{p+q}_{\mathbb{C}} \cong C^{p,q}_{\mathbb{C}}$ where $C^{p,q}_{\mathbb{C}}:=C^{p,q}_{\mathbb{R}} \underset{\mathbb{R}}{\otimes} \mathbb{C}$.

\subsection{Algebras over $\mathbb{C}$ with antilinear automorphism} \label{subsection algebras antilinear}

The field of real numbers can be also understood as the conjugation invariant field of the complex numbers. With this approach we could
see the category of finite dimensional vector spaces over $\mathbb{R}$
as equivalent to the category of finite dimensional vector spaces
over $\mathbb{C}$ endowed with a $\mathbb{C}$-antilinear involution.

If $\mathrm{Vect}(\mathbb{R})$ denotes the category of finite dimensional $\mathbb{R}$-vector spaces, and $\mathrm{Vect}_{(\mathbb{Z}/2,\mathrm{id})}(\mathbb{C})$
denotes the category of finite dimensional $\mathbb{C}$-vector spaces
endowed with an action of $\mathbb{Z}/2$ whose linearity is 
defined by the homomorphism $\mathrm{id}:\mathbb{Z}/2 \to \mathrm{Gal}(\mathbb{C}/\mathbb{R})\cong \mathbb{Z}/2$, then Speiser's lemma \cite[Lem. 2.3.8]{Gille_Central_Simple_algebras} implies that 
the functors:
\begin{align}
    \mathrm{Vect}(\mathbb{R}) & \to \mathrm{Vect}_{(\mathbb{Z}/2,\mathrm{id})}(\mathbb{C}) &  \mathrm{Vect}_{(\mathbb{Z}/2,\mathrm{id})}(\mathbb{C}) & \to \mathrm{Vect}(\mathbb{R})\\
    V & \mapsto V \underset{\mathbb{R}}{\otimes} \mathbb{C} &  
    W & \mapsto W^{\mathbb{Z}/2}
\end{align}
provides the equivalence of the categories.

The complex tensor product makes $\mathrm{Vect}_{(\mathbb{Z}/2,\mathrm{id})}(\mathbb{C})$
into a tensor category, and an algebra object in this category
is nothing else but an algebra $A$ over the complex numbers endowed with a $\mathbb{C}$-antilinear automorphism $\tau : A \overset{\cong}{\to} A$. Two complex algebras 
endowed with $\mathbb{C}$-antilinear automorphisms are isomorphic if they
are isomorphic as complex algebras and the isomorphism is equivariant
with respect to the automorphisms.

We call graded algebras over $\mathbb{C}$ with antilinear automorphism
 GA$_{(\mathbb{Z}/2,\mathrm{id})}(\mathbb{C})$, and accordingly, we
may take the central and simple complex algebras with $\mathbb{C}$-antilinear automorphism CSGA$_{(\mathbb{Z}/2,\mathrm{id})}(\mathbb{C})$.

Two algebras $(A,\tau_A)$ and $(B, \tau_B)$ in CSGA$_{(\mathbb{Z}/2,\mathrm{id})}(\mathbb{C})$ are similar
if there is an isomorphism of complex algebras:
\begin{align}
    A \hat{\otimes}\mathrm{End}(V_0,V_1) \cong B \hat{\otimes}\mathrm{End}(W_0,W_1)
\end{align}
compatible with the induced $\mathbb{C}$-antilinear automorphisms on each side. Here the automorphism on the endomorphism algebra
$\mathrm{End}(V_0,V_1)$ is given by the adjoint map of the complex conjugation:
\begin{align}
    \mathrm{Ad}_{\mathbb{K}} : \mathrm{End}(V_0,V_1) & \to \mathrm{End}(V_0,V_1) \\
    \rho & \mapsto \mathbb{K} \rho \mathbb{K},
\end{align}
 which on vectors act in the following form  $v \mapsto \mathbb{K} \rho(\mathbb{K} v)$. Here $\mathbb{K}$ denotes complex conjugation.

The equivalence classes of similar algebras in CSGA$_{(\mathbb{Z}/2,\mathrm{id})}(\mathbb{C})$ is the Graded Brauer group and will be denoted
as:
\begin{align}
    \mathrm{GrBr}_{(\mathbb{Z}/2,\mathrm{id})}(\mathbb{C}).
\end{align}

The following proposition follows from Speiser's lemma \cite[Lem. 2.3.8]{Gille_Central_Simple_algebras} and is a key result in our generalization of the Graded Brauer group for magnetic groups.

\begin{proposition}
    There is a canonical isomorphism of groups $\mathrm{GrBr}(\mathbb{R}) \cong \mathrm{GrBr}_{(\mathbb{Z}/2,\mathrm{id})}(\mathbb{C})$ realized by the maps:
    \begin{align}
    \mathrm{GrBr}(\mathbb{R}) & \to \mathrm{GrBr}_{(\mathbb{Z}/2,\mathrm{id})}(\mathbb{C}) &  \mathrm{GrBr}_{(\mathbb{Z}/2,\mathrm{id})}(\mathbb{C}) & \to \mathrm{GrBr}(\mathbb{R})\\
    A & \mapsto (A \underset{\mathbb{R}}{\otimes} \mathbb{C}, \mathrm{Ad}_{\mathrm{id} \otimes \mathbb{K}}) &  
    (A, \tau) & \mapsto A^{\tau}.
\end{align}
Here $A^{\tau}$ denotes the fixed points of the automorphism and $\mathbb{K}$ denotes complex conjugation.
\end{proposition}

\begin{proof}
The natural map $\lambda : A^\tau \underset{\mathbb{R}}{\otimes} \mathbb{C} \to A$, $\lambda(a \otimes \mu) = \mu a$, is a $\mathbb{C}$-linear map
that provides the isomorphism of complex vector spaces; this is
Speiser's lemma \cite[Lem. 2.3.8]{Gille_Central_Simple_algebras}. 
This isomorphism is equivariant with the $\mathbb{C}$-antilinear isomorphisms of the algebras:
\begin{align}
\lambda (\mathrm{Ad}_{\mathrm{id}\otimes \mathbb{K}}(a\otimes\mu))=
    \lambda(\rm{id} \otimes \mathbb{K}(a \otimes \mu)\rm{id} \otimes \mathbb{K})= \lambda(a \otimes \bar{\mu}) = \bar{\mu} a = \tau(\mu a) = \tau(\lambda(a \otimes \mu)).
\end{align}
\end{proof}

The similarity classes of real Clifford algebras $C^{p,q}_{\mathbb{R}}$ 
generate the group $\mathrm{GrBr}(\mathbb{R}) \cong \mathbb{Z}/8$ and they can be rewritten in terms
of complex central algebras with antilinear automorphisms. Let us show explicit generators of these 8 similarity classes:

\begin{itemize}
    \item[$\boxed{p+q=0}$] The Clifford algebra $C^{0,0}_{\mathbb{R}}$ is $\mathbb{R}$ and it is mapped to the algebra $\mathbb{C}$ with  complex conjugation as antilinear automorphism $\tau^{0,0}=\mathrm{Ad}_\mathbb{K}$. We therefore have:
    \begin{align}
        C^{0,0}_{\mathbb{R}} \mapsto (\mathbb{C}, \tau^{0,0}=\mathrm{Ad}_\mathbb{K})
    \end{align}
    \item[$\boxed{p+q=1}$] Both Clifford algebras $C^{1,0}_{\mathbb{R}}$ and $C^{0,1}_{\mathbb{R}}$ map to the complex algebra: 
    \begin{align}    
    C^{1}_{\mathbb{C}} \cong \mathbb{C}[e] /\langle e^2=1 \rangle
    \end{align}
    where the automorphism $\tau^{1.0}$ associated to $C^{1,0}_{\mathbb{R}}$
    is simply complex conjugation 
    $\tau^{1,0}(x+ye)=\bar{x}+\bar{y}e$,
     while 
    the automorphism $\tau^{0,1}$ associated to $C^{0,1}_{\mathbb{R}}$
    is 
    $\tau^{1,0}(x+ye)=\bar{x}-\bar{y}e$.
    In the former case the element $e$ is 
    $\tau^{1,0}$-invariant with $e^2=1$, and in the latter $ie$ is   $\tau^{0,1}$-invariant with $(ie)^2=-1.$ We therefore have:
    \begin{align} \label{complexification C10}
       C^{1,0}_{\mathbb{R}} \mapsto & \left( \mathbb{C}[e] /\langle e^2=1 \rangle, \tau^{1,0}=\mathrm{Ad}_\mathbb{K}\right), \\
       C^{0,1}_{\mathbb{R}} \mapsto & \left( \mathbb{C}[e] /\langle e^2=-1 \rangle, \tau^{0,1}=\mathrm{Ad}_\mathbb{K}\right) 
    \end{align}
     
    \item[$\boxed{p+q=2}$] The three Clifford algebras $C^{2,0}_{\mathbb{R}}$, $C^{1,1}_{\mathbb{R}}$ and $C^{0,2}_{\mathbb{R}}$
    map to the complex algebra:
    \begin{align}    
    C^{2}_{\mathbb{C}} \cong M_{1|1}(\mathbb{C}).
    \end{align}
    The automorphisms $\tau^{p,q}$ in terms of matrices are: 
    \begin{align}
        \tau^{2,0} \begin{pmatrix} \alpha & \beta \\ \gamma & \delta 
        \end{pmatrix} = \begin{pmatrix} \bar{\delta} & \bar{\gamma} \\ \bar{\beta} & \bar{\alpha}
        \end{pmatrix} \ \ \ \ \ 
        \tau^{1,1} \begin{pmatrix} \alpha & \beta \\ \gamma & \delta 
        \end{pmatrix} = \begin{pmatrix} \bar{\alpha} & \bar{\beta} \\ \bar{\gamma} & \bar{\delta}
        \end{pmatrix} \ \ \ \ \ 
        \tau^{0,2} \begin{pmatrix} \alpha & \beta \\ \gamma & \delta 
        \end{pmatrix} = \begin{pmatrix} \bar{\delta} & -\bar{\gamma} \\ -\bar{\beta} & \bar{\alpha}
        \end{pmatrix} 
    \end{align}
    and in terms of operators are: 
    \begin{align}
        \tau^{2,0} =  \mathrm{Ad}_{\left(\begin{smallmatrix}
            0 & 1 \\ 1 & 0
        \end{smallmatrix}\right) \mathbb{K}},  \ \ \ 
        \tau^{1,1} = \mathrm{Ad}_\mathbb{K}, \ \ \ 
        \tau^{0,2} = \mathrm{Ad}_{\left(\begin{smallmatrix}
            0 & -1 \\ 1 & 0
        \end{smallmatrix}\right) \mathbb{K}}.
    \end{align}
Here the adjoint operator $\mathrm{Ad}_{M\mathbb{K}}$ means $\mathrm{Ad}_{M\mathbb{K}}(C)=M\mathbb{K}C \mathbb{K}M^{-1}$
for $C \in \mathrm{End}(\mathbb{C}, \mathbb{C})$. We therefore have the assignment:
\begin{align} \label{C{2,0} to complex}
    C^{2,0}_{\mathbb{R}} \mapsto & \left( M_{1|1}(\mathbb{C}), \tau^{2,0} =  \mathrm{Ad}_{\left(\begin{smallmatrix}
            0 & 1 \\ 1 & 0
        \end{smallmatrix}\right)    \mathbb{K}}\right)\\
    C^{1,1}_{\mathbb{R}} \mapsto & \left( M_{1|1}(\mathbb{C}), \tau^{1,1} = \mathrm{Ad}_\mathbb{K} \right) \\
        C^{0,2}_{\mathbb{R}} \mapsto & \left( M_{1|1}(\mathbb{C}), \tau^{0,2} =  \mathrm{Ad}_{\left(\begin{smallmatrix}
            0 & -1 \\ 1 & 0
    \end{smallmatrix}\right) \mathbb{K}} \right) 
\end{align}
\item[$\boxed{p+q=3}$] Here we are interested in the Clifford algebras 
$C^{3,0}_{\mathbb{R}}$ and $C^{0,3}_{\mathbb{R}}$. In both cases, the 
algebra generated by the elements $e_1e_2$ and $e_2e_3$ generate the degree 
0 part of the algebras and this is isomorphic to the quaternions 
$\mathbb{H}$. If we further add the element $e_1e_2e_3$ of degree one, we see 
that it commutes with with $e_1e_2$ and $e_2e_3$ and it squares to $-1$ in
$C^{3,0}_{\mathbb{R}}$ and to $1$ in $C^{3,0}_{\mathbb{R}}$. Therefore we have the graded algebra isomorphisms:
\begin{align}
C^{3,0}_{\mathbb{R}} \cong \mathbb{H}[e] / \langle e^2=-1 \rangle,  \ \ \ 
    C^{0,3}_{\mathbb{R}} \cong \mathbb{H}[e] / \langle e^2=1 \rangle.
\end{align}
After complexification, the quaternions algebra $\mathbb{H}$ becomes the matrix algebra $M_2(\mathbb{C})$ with antiunitary automorphism given 
by the operator$ \mathrm{Ad}_{\left(\begin{smallmatrix}
            0 & -1 \\ 1 & 0
        \end{smallmatrix}\right) \mathbb{K}}$,
and the Clifford algebras $C^{3,0}_{\mathbb{R}}$ and $C^{0,3}_{\mathbb{R}}$ are mapped to the complex algebra $M_2(\mathbb{C})[e]/\langle e^2 =1 \rangle$. If we write the elements in this algebra as $(M_0 + M_1e)$ with $M_i \in M_2(\mathbb{C})$, then the Clifford algebras transform as follows:
\begin{align}
    C^{3,0}_\mathbb{R} \mapsto & \left(M_2(\mathbb{C})[e]/\langle e^2 =-1 \rangle, \tau^{3,0}= \mathrm{Ad}_{\left(\begin{smallmatrix}
            0 & -1 \\ 1 & 0
        \end{smallmatrix}\right)  \mathbb{K}}  \right),\\
        C^{0,3}_\mathbb{R} \mapsto & \left(M_2(\mathbb{C})[e]/\langle e^2 =1 \rangle, \tau^{3,0}=  \mathrm{Ad}_{\left(\begin{smallmatrix}
            0 & -1 \\ 1 & 0
        \end{smallmatrix}\right)  \mathbb{K}}  \right).
\end{align}

\item[$\boxed{p+q=4}$] Here we are interested in $C^{0,4}_{\mathbb{R}}$
which can be seen alternatively as:
\begin{align}
    C^{0,4}_{\mathbb{R}} \cong C^{0,3}_{\mathbb{R}} \hat{\otimes} C^{0,1}_{\mathbb{R}}  \cong (\mathbb{H} \otimes  C^{1,0}_{\mathbb{R}} )\hat{\otimes } C^{0,1}_{\mathbb{R}} \cong \mathbb{H} \otimes ( C^{1,0}_{\mathbb{R}} \hat{\otimes } C^{0,1}_{\mathbb{R}}) \cong \mathbb{H} \otimes \mathrm{End}(\mathbb{R},\mathbb{R}).
\end{align}
Here we obtain the well-known fact that the graded Clifford algebra $C^{0,4}_{\mathbb{R}}$ is similar to the quaternions algebra $\mathbb{H}$
localized in degree zero.

After complexification we obtain:
\begin{align}
      C^{0,4}_{\mathbb{R}} \underset{\mathbb{R}}{\otimes} \mathbb{C} \cong 
      (\mathbb{H} \underset{\mathbb{R}}{\otimes} \mathbb{C}) \otimes( M_{1|1}(\mathbb{R})  \underset{\mathbb{R}}{\otimes}\mathbb{C} ) \cong 
      M_2(\mathbb{C}) \otimes M_{1|1}(\mathbb{C}),
\end{align}
where $M_2(\mathbb{C})$ denotes the matrix algebra over the complex numbers
localized in degree zero. The associated antilinear automorphism is simply the operator $\tau^{0,4} =  \mathrm{Ad}_{\left(\begin{smallmatrix}
            0 & -1 \\ 1 & 0
        \end{smallmatrix}\right)  \mathbb{K}} \otimes \mathrm{Ad}_\mathbb{K} $ 
        where the adjoint operator is defined with a matrix of homogeneous degree zero, and the complex conjugation on $M_{1|1}(\mathbb{C})$ is also of degree zero. So we have:
\begin{align} \label{C04 clifford}
         C^{0,4}_{\mathbb{R}} \mapsto  &\left(M_2(\mathbb{C}) \otimes M_{1|1}(\mathbb{C}) , \tau^{0,4} =  \mathrm{Ad}_{\left(\begin{smallmatrix}
            0 & -1 \\ 1 & 0
        \end{smallmatrix}\right)  \mathbb{K}} \otimes \mathrm{Ad}_\mathbb{K} \right),\\
         \mathbb{H} \mapsto & \left(M_2(\mathbb{C}),  \mathrm{Ad}_{\left(\begin{smallmatrix}
            0 & -1 \\ 1 & 0
        \end{smallmatrix}\right)\mathbb{K}}\right) . \label{M2(C) generator Br(R)}
\end{align}        
\end{itemize}

\subsection{Graded division algebras and the tenfold way} The graded version of the famous 
Wedderburn-Artin theorem states that any graded simple algebra over the 
field $k$ is isomorphic to a graded matrix algebra with coefficients
in a {\it graded division algebra} over $k$ \cite[Thm. 6.2.5]{Varadarajan_Supersymmetry}. 
A {\it graded division algebra} over $k$ is a graded algebra where every non-zero
homogeneous element is invertible. 

Frobenius showed \cite{Frobenius} that
there are only three division algebras over $\mathbb{R}$, being them
$\mathbb{R}$, $\mathbb{C}$ and $\mathbb{H}$. Wall\cite{Wall_Graded_Brauer_groups}, and later Deligne \cite[\S 3]{Deligne_Notes_Spinors}, showed that
there are ten graded division algebras over $\mathbb{R}$, three of pure even degree $\mathbb{R}$, $\mathbb{C}$ and $\mathbb{H}$, and seven Clifford 
graded algebras, six real Clifford algebras $C^{3,0}_\mathbb{R}$, $C^{2,0}_\mathbb{R}$, 
$C^{1,0}_\mathbb{R}$, $C^{0,1}_\mathbb{R}$, $C^{0,2}_\mathbb{R}$,and $C^{0,3}_\mathbb{R}$, and one complex Clifford algebra
$C^{1}_\mathbb{C}$. The fact that there are only ten graded division algebras over $\mathbb{R}$ is nice to prove, and an elegant proof has been showed by Baez \cite{Baez_tenfold}. This is our first tenfold classification; let us summarize it.

A graded division algebra $A=A_0\oplus A_1$ over $\mathbb{R}$ 
has for degree zero an algebra $A_0$ whose non-zero elements are 
invertible, hence a division algebra. Therefore $A_0$ is $\mathbb{R}$,
$\mathbb{C}$ or $\mathbb{H}$. 
\begin{align}
    \mathrm{If} \ A_1=\{0\}, \mathrm{then} \ A \ \mathrm{is} \ \mathbb{R}, \mathbb{C} \ \mathrm{or} \ \mathbb{H}.
\end{align}
 If $A_1 \neq \{0\}$ then take $e \in A_1$ 
and note that $e$ is invertible, and moreover $e^2 \in A_0$. 
In the case that $A_0=\mathbb{R}$ we may rescale $e$ in order to obtain $e^2 =1$
or $e^2=-1$. 
\begin{align}
    \mathrm{If} \ A_0=\mathbb{R} \ \mathrm{and} \ A_0=\{0\}, \mathrm{then} \ A \ \mathrm{is}  \left\{ \begin{array}{l}\mathbb{R}[e]/\langle e^2= 1 \rangle, \ \mathrm{or}   \\ \mathbb{R}[e]/\langle e^2=- 1 \rangle \end{array} \right.
\end{align}

If $A_0 = \mathbb{C}$ then the conjugation map $a \mapsto eae^{-1}$ is 
an automorphism of $A_0$, and therefore it is the identity or complex conjugation. In the latter case we have that $ea=ae$ and we may rescale $e$ such that $e^{2}=1$. In the former case we have $ea=\overline{a}e$, and we may rescale $e$ such that either $e^2=1$ or $e^2=-1$.
\begin{align}
    \mathrm{If} \ A_0=\mathbb{C} \ \mathrm{and} \ A_0=\{0\}, \mathrm{then} \ A \ \mathrm{is} \left\{ \begin{array}{l} \mathbb{C}[e]/\langle e^2= 1 \rangle \ \mathrm{with} \ ei=ie,\ \mathrm{or}\\
     \mathbb{C}[e]/\langle e^2= 1 \rangle
    \ \mathrm{with} \ ei=-ie,\ \mathrm{or}\\
     \mathbb{C}[e]/\langle e^2= -1 \rangle
    \ \mathrm{with} \ ei=-ie.
    \end{array} \right.
\end{align}

If $A_0 = \mathbb{H}$ then the conjugation map $a \mapsto eae^{-1}$ is 
an automorphism of $A_0$, and all automorphisms of the quaternions are inner
transformations. This means that there must be a quaternion $q$ such that
$qaq^{-1}=eae^{-1}$, and therefore $q^{-1}e$ commutes with $A_0$. Replacing
$q^{-1}e$ by $e$, we see that $e$ commutes with $A_0$ and therefore $e^2$
also commutes with $A_0$. Since $e^2$ belongs to $A_0$, we see that $e \in \mathbb{R} = Z(\mathbb{H})$. Rescaling $e$ we see that $e^2=1$ or $e^2=-1$.
\begin{align}
    \mathrm{If} \ A_0=\mathbb{H} \ \mathrm{and} \ A_0=\{0\}, \mathrm{then} \ A \ \mathrm{is} \left\{ \begin{array}{l} \mathbb{H}[e]/\langle e^2= 1 \rangle \ \mathrm{with} \ ei=ie, ej=je,\ \mathrm{or}\\
     \mathbb{H}[e]/\langle e^2= -1 \rangle
    \ \mathrm{with} \ ei=ie, ej=je.
    \end{array} \right.
\end{align}

\begin{table}[!h]
\begin{center}
\caption{Tenfold way in terms of the ten real division algebras,
and their relation to the graded Brauer group of the reals and the complex numbers, the graded Brauer group of central simple complex algebras endowed with an antilinear automorphism, and the Altland-Zirnbauer classification of non-interacting fermionic Hamiltonians. The first column corresponds to a numbering
that is compatible with the group structure of the Graded Brauer groups. The second column corresponds to the a choice of 
representatives of the elements in the graded Brauer group over the reals and the complex numbers. The third column is a description of the division algebra that exists on each element of the Brauer groups, The fourth column is a choice of representative for the elements in the graded Brauer group of central simple complex algebras with antilinear automorphism. The last five columns provide the associate 
Altland-Zirnbauer  classification scheme of Hamiltonian symmetries.
The operators $\mathcal{T}$, $\mathcal{C}$ and $\mathcal{S}$ are respectively time reversal, particle hole and chiral symmetry. A zero represents absence of the symmetry while a $\pm 1$ encodes the square of the operator which is preserved. The associated operator is shown in the respective column.
}
\label{Table-tenfold-way}
\begin{tabular}{|c | c |c | c|| c| c|c|c|c|} 
 \hline
 N. & $\mathrm{GrBr}(\mathbb{R})$ & DivAlg$(\mathbb{R})$ & 
 $\mathrm{GrBr}_{(\mathbb{Z}/2,\mathrm{id})}(\mathbb{C})$  & $\mathcal{T}$ & $\mathcal{C}$ & $\mathcal{S}$ & $\mathcal{T}\mathcal{C}\mathcal{S}$  &AZ  \\ [0.5ex] 
 \hline \hline
 0 & $C^{0,0}_\mathbb{R}$ & $\mathbb{R}$ &
 $(\mathbb{C},\mathrm{Ad}_\mathbb{K})$
 & $\mathbb{K}$ & 0& 0&  1|0|0 & AI \\ 
 \hline
 1  & $C^{1,0}_\mathbb{R}$ & $\frac{\mathbb{R}[e]}{\langle e^2 =1 \rangle}$  &
 $ \left( \frac{\mathbb{C}[e]}{\langle e^2=1 \rangle}, \mathrm{Ad}_\mathbb{K}\right)$
 & $\mathbb{K}$& $\mathbb{K}e$& $e$& 1|1|1& BDI  \\
 \hline  
2 & $C^{2,0}_\mathbb{R}$ & $\frac{\mathbb{C}[e]}{\langle e^2 =1,  ei=-ie \rangle} $ & 
$\left(M_{1|1}(\mathbb{C}),   \mathrm{Ad}_{\left(\begin{smallmatrix}
            0 & 1 \\ 1 & 0
        \end{smallmatrix}\right)    \mathbb{K}}\right)$
& 0& ${\left(\begin{smallmatrix}
            0 & 1 \\ 1 & 0
        \end{smallmatrix}\right)  \mathbb{K}}$& 0& 0|1|0&  D \\
 \hline  
3 & $C^{3,0}_\mathbb{R}$ & $\frac{\mathbb{H}[e]}{\langle e^2 = -1 , ei=ie, ej=je\rangle}$ &
$\left(\frac{M_2(\mathbb{C})[e]}{\langle e^2 =-1 \rangle},  \mathrm{Ad}_{\left(\begin{smallmatrix}
            0 & -1 \\ 1 & 0
        \end{smallmatrix}\right)  \mathbb{K} } \right)$
& ${\left(\begin{smallmatrix}
            0 & -1 \\ 1 & 0
        \end{smallmatrix}\right)  \mathbb{K}}$& ${\left(\begin{smallmatrix}
            0 & -1 \\ 1 & 0
        \end{smallmatrix}\right)  \mathbb{K}e}$& $ie$&  -1|1|1 & DIII\\
 \hline  
4 & $C^{0,4}_\mathbb{R}$ & $\mathbb{H}$ &
$\left(M_2(\mathbb{C}),  \mathrm{Ad}_{\left(\begin{smallmatrix}
            0 & -1 \\ 1 & 0
        \end{smallmatrix}\right)  \mathbb{K}}\right)$
& ${\left(\begin{smallmatrix}
            0 & -1 \\ 1 & 0
        \end{smallmatrix}\right)  \mathbb{K}}$& 0& 0& -1|0|0&  AII\\
 \hline 
5 & $C^{0,3}_\mathbb{R}$ & $\frac{\mathbb{H}[e]}{\langle e^2 = 1, ei=ie, ej=je \rangle}$&
$\left(\frac{M_2(\mathbb{C})[e]}{\langle e^2 =1 \rangle},  \mathrm{Ad}_{\left(\begin{smallmatrix}
            0 & -1 \\ 1 & 0
        \end{smallmatrix}\right)  \mathbb{K}}  \right)$
& ${\left(\begin{smallmatrix}
            0 & -1 \\ 1 & 0
        \end{smallmatrix}\right)  \mathbb{K}}$& ${\left(\begin{smallmatrix}
            0 & -1 \\ 1 & 0
        \end{smallmatrix}\right)  \mathbb{K}e}$& $e$ & -1|-1|1&  CII\\
 \hline  
6 & $C^{0,2}_\mathbb{R}$ &  $\frac{\mathbb{C}[e]}{\langle e^2 =-1, ei=-ie \rangle}$&
$\left( M_{1|1}(\mathbb{C}),  \mathrm{Ad}_{\left(\begin{smallmatrix}
            0 & -1 \\ 1 & 0
        \end{smallmatrix}\right)   \mathbb{K}}\right)$
& 0&${\left(\begin{smallmatrix}
            0 & -1 \\ 1 & 0
        \end{smallmatrix}\right)  \mathbb{K}}$ &0 & 0|-1|0 & C \\
 \hline 
7 & $C^{0,1}_\mathbb{R}$ & $\frac{\mathbb{R}[e]}{\langle e^2 = -1 \rangle}$  &
$ \left( \frac{\mathbb{C}[e]}{\langle e^2=-1 \rangle}, \mathrm{Ad}_\mathbb{K}\right)$
& $\mathbb{K}$ &$\mathbb{K}e$ & $ie$& 1|-1|1 & CI \\
 \hline  \hline
0 & $C^0_{\mathbb{C}}$ & $\mathbb{C}$ & $\mathbb{C}$ & 0& 0&0 & 0|0|0& A\\
 \hline  
1 & $C^1_{\mathbb{C}}$ & $\frac{\mathbb{C}[e]}{\langle e^2 =1, ei=ie \rangle} $ & $\frac{\mathbb{C}[e]}{\langle e^2 =1 \rangle}$ & 0& 0& $e$ &  0|0|1&AIII\\
 \hline \hline
  N. & $\mathrm{GrBr}(\mathbb{C})$ & DivAlg$(\mathbb{R})$ & 
 $\mathrm{GrBr}(\mathbb{C})$  & $\mathcal{T}$ & $\mathcal{C}$ & $\mathcal{S}$ & $\mathcal{T}\mathcal{C}\mathcal{S}$ & AZ  \\
 \hline
\end{tabular}
\end{center}
\end{table}

Another tenfold classification structure was shown by Dyson \cite{Dyson_Threefold} when studying the symmetries of Hamiltonianis
in nuclear physics, where the threefold classification of complex linear representations was generalized to a tenfold classification whenever
the representations possess both complex linear and complex antilinear
symmetry transformations. A more recent appearance of this tenfold
structure
is the Altland-Zirnbauer (AZ) scheme classification of free-fermion Hamiltonians \cite{Altland}, a classification which was masterfully rewritten
by Kitaev \cite{kitaev}
in terms of the 2 types of complex Clifford algebras and the 8 types of
real Clifford algebras with the name of "Periodic table for topological insulators and superconductors".

Let us recall the basic ideas in the AZ classification \cite{Altland} of non-interacting fermionic Hamiltonians in terms of ten symmetry classes \cite{FreedMoore_Twisted_equivariant_matter}. These are classified in terms of the presence or
absence of certain discrete symmetries on a prescribed Hamiltonian \cite[\S II]{Schnyder_Classification-of-topological-insulators}.

Suppose that this Hamiltonian encodes the electromagnetic properties of
some crystal and take its eigenvectors and its spectrum. Assume that the Fermi level (energy of occupied bands) is set at zero and denote
by {\it valence bands } the vector space generated by the negative energy eigenvectors, and {\it conduction bands} the vector space generated by
positive energy eigenvectors. 

Assume furthermore that the Hamiltonian is gapped at the Fermi level,
thus modeling an insulating material, and permitting to have a well defined
vector space of valence and conduction bands.

Assign degree zero to the valence bands and degree one to the conduction bands.
This way the vector space spanned by the eigenvectors of the Hamiltonian becomes a graded vector space.

A Hamiltonian preserves {\it time reversal symmetry} (TRS) if there is a
$\mathbb{C}$-antilinear operator $\mathcal{T}$, such that:
\begin{align}
    \mathcal{T} \mathcal{H} \mathcal{T}^{-1} = \mathcal{H}, \ \mathrm{with}
    \ \mathcal{T}= U_{\mathcal{T}} \mathbb{K} \ \mathrm{and} \ \mathcal{T}^{2}=\pm 1
\end{align}
where $U_{\mathcal{T}} $ is a unitary operator ($U_{\mathcal{T}}^{-1}=U_{\mathcal{T}}^\dagger$). The fact that the operator is $\mathbb{C}$-antilinear induces two non-equivalent possibilities for the operator $\mathcal{T}^2$ as was shown by Wigner \cite{wigner}.
If the Hamiltonian preserves TRS, then the TRS operator $\mathcal{T}$
acts on the graded vector space spanned by the eigenvectors of the Hamiltonian as a degree zero operator. 

A Hamiltonian preserves {\it particle hole symmetry} (PHS) if there is a
$\mathbb{C}$-antilinear operator $\mathcal{C}$, such that:
\begin{align}\label{PHS}
    \mathcal{C} \mathcal{H} \mathcal{C}^{-1} = - \mathcal{H}, \ \mathrm{with}
    \ \mathcal{C}= U_{\mathcal{C}} \mathbb{K} \ \mathrm{and} \ \mathcal{C}^{2}=\pm 1
\end{align}
where $U_{\mathcal{C}} $ is a unitary operator ($U_{\mathcal{C}}^{-1}=U_{\mathcal{C}}^\dagger$). If the Hamiltonian preserve PHS, then the operator $\mathcal{C}$ exchanges occupied states with  unoccupied ones, or particles with antiparticles (holes), inverting the charge.
It changes the sign of the energy of an eigenstate of the Hamiltonian, and therefore acts on the graded vector space spanned by the eigenvectors of the Hamiltonian as a degree one operator. 

If a Hamiltonian preserves both TRS and PHS, then it preserves the composition $\mathcal{S}=\mathcal{T}\mathcal{C}$. This is a $\mathbb{C}$-linear transformation, which after rescaling can be made to satisfy $\mathcal{S}^2=1$. This symmetry is called {\it chiral symmetry} and satisfies:
\begin{align}
    \mathcal{S} \mathcal{H} \mathcal{S}^{-1} = - \mathcal{H}, \ \mathrm{with}
    \ \mathcal{S}= U_{\mathcal{S}} \ \mathrm{and} \ \mathcal{S}^{2}= 1
\end{align}
with $U_{\mathcal{S}}$ a unitary operator ($U_{\mathcal{S}}^{-1}= U_{\mathcal{S}}^\dagger$). The chiral symmetry operator $\mathcal{S}$ becomes a degree one operator
exchanging the valence and the conduction bands without inverting the charge.  Note that a Hamiltonian
may possess chiral symmetry without possessing TRS or PHS. 

The ten different combinations of the three symmetries is known as the AZ
classification scheme and it appears in Table \ref{Table-tenfold-way}. If a symmetry is not present, a $0$ appears, and
if a symmetry is present, a $\pm 1$ appears characterizing the square of the operator.

This tenfold classification could be understood as an alternative appearance
of the fact that there are only ten real division algebras. Since
the operators $\mathcal{T}$,  $\mathcal{C}$ and  $\mathcal{S}$
act on the complex vector space of eigenvectors of the Hamiltonian, we could
focus on the complex algebra that they generate as operators. 
Comparing the modules over these algebras with the modules of the algebras that generate the Graded Brauer groups $\mathrm{GrBr}(\mathbb{C})$ and $\mathrm{GrBr}_{(\mathbb{Z}/2, \mathrm{id})}(\mathbb{C})$, we find the following set of matchings:
\begin{itemize}
    \item[\boxed{\mathrm{A}}] No symmetries. Corresponds to the algebra $\mathbb{C}$.
    \item[\boxed{\mathrm{AIII}}] Corresponds to the graded algebra $\mathbb{C}[e]/ \langle e^2=1 \rangle$ where $\mathcal{S}$ is mapped to the operator $e$.
    \item[\boxed{\mathrm{AI}}] Corresponds to the algebra $\mathbb{C}$ with  antilinear automorphism $\mathbb{K}$. Here $\mathbb{T}$ is mapped to $\mathbb{K}$.
    \item[\boxed{\mathrm{AII}}] Corresponds to the algebra $M_2(\mathbb{C})$ with antilinear automorphism given by the adjoint of the matrix $\left(\begin{smallmatrix}
            0 & -1 \\ 1 & 0
        \end{smallmatrix}\right)$ with $\mathbb{K}$. Here $\mathcal{T}$ is mapped to the composition of operators $\left(\begin{smallmatrix}
            0 & -1 \\ 1 & 0
        \end{smallmatrix}\right) \mathbb{K}$.
    \item[\boxed{\mathrm{BDI}\ \& \ \mathrm{CI}}] Corresponds to the graded algebras $\mathbb{C}[e]/\langle e^2= \pm 1 \rangle$ with antilinear automorphism $\mathbb{K}$. Here
    $\mathcal{T}$ is mapped to $\mathbb{K}$ and $\mathcal{C}$ is mapped to $e\mathbb{K}$.
    \item[\boxed{\mathrm{D}\ \& \ \mathrm{C}}]  Corresponds to the graded algebra $M_{1|1}(\mathbb{C})$ with antilinear automorphism given by the adjoint of the degree one operator $\left(\begin{smallmatrix}
            0 & \pm 1 \\ 1 & 0
        \end{smallmatrix}\right)$ composed with $\mathbb{K}$. Here $\mathcal{C}$ is mapped to $\left(\begin{smallmatrix}
            0 & \pm 1 \\ 1 & 0
        \end{smallmatrix}\right)\mathbb{K}$.
        \item[\boxed{\mathrm{DIII}\ \& \ \mathrm{CII}}]  Corresponds to the algebras $M_2(\mathbb{C})[e] / \langle e^2= \pm 1 \rangle$ with antilinear automorphism given by the adjoint of $\left(\begin{smallmatrix}
            0 & - 1 \\ 1 & 0
        \end{smallmatrix}\right)$ composed with $\mathbb{K}$. Here $\mathcal{T}$ is mapped to the operator $\left(\begin{smallmatrix}
            0 & - 1 \\ 1 & 0
        \end{smallmatrix}\right) \mathbb{K}$  and $\mathcal{C}$ is mapped
        to the operator $\left(\begin{smallmatrix}
            0 & - 1 \\ 1 & 0
        \end{smallmatrix}\right)  \mathbb{K}e$.
 \end{itemize}





\section{Magnetic Equivariant Graded Brauer group}

In the study of the electromagnetic properties
of periodic free-fermion Hamiltonians associated to crystals, the  crystal symmetries play a fundamental role. Whenever the crystal has a non-trivial internal magnetization, there might be symmetries of the Hamiltonian which are the composition of either TRS, PHS or chiral symmetry with the geometrical symmetries of the crystal. This gives rise to the concept of  magnetic groups, where the elements of the group may act complex linearly or complex antilinearly \cite{serrano2025magneticequivariantktheory}.

In this section we are going to study magnetic equivariant simple central complex algebras and we are going to explicitly calculate its similarity classes. This group will be called the {\it{magnetic equivariant graded Brauer group.}}

Now, the magnetic and equivariant version of the Graded Brauer group incorporates results of magnetic and equivariant projective representations. Therefore, prior to studying the Brauer group we will
review some results in the theory of non abelian cohomology and in the theory of magnetic projective representations. The we will introduce the magnetic equivariant graded Brauer group associated to a magnetic group, and we provide an explicit calculation of this group in terms of cohomology of groups; this is Thm. \ref{Theorem decomposition GrBr} and it is the main result of this section.
Then we will relate this Brauer group with the graded equivariant Brauer groups of real and complex algebras and we calculate some explicit examples. 

\subsection{Magnetic groups and magnetic equivariant representations}
A magnetic group consist of a pair $(G, \phi)$ where $G$ is a group and $\phi : G \to \mathbb{Z}/2$ is a surjective homomorphism. Morphisms of magnetic groups $(G, \phi) \to (G' ,\phi')$ consist of homomorphisms of groups $f: G \to G'$ which are compatible with the projections to $\mathbb{Z}/2$, ie. $ \phi' \circ f = \phi$.  
In this work we will focus our attention only to finite magnetic groups.

The magnetic general linear group $\mathrm{MGL}_n(\mathbb{C})$ is the semidirect product: 
\begin{align}
    \mathrm{MGL}_n(\mathbb{C}) := \mathrm{GL}_n(\mathbb{C}) \rtimes \mathbb{Z}/2
\end{align}
where $\mathrm{GL}_n(\mathbb{C})$ is the general linear group of invertible $n$x$n$ complex matrices, and
 $\mathbb{Z}/2$ acts by complex conjugation on the matrices of $\mathrm{GL}_n(\mathbb{C})$. Explicitly the product is:
\begin{align}
    (\Lambda_0,\tau_0) (\Lambda_1,\tau_1) =  (\Lambda_0\ {\mathbb{K}}^{\tau_0}\Lambda_1,\tau_0\tau_1), \ \ \ \ \
    {\mathbb{K}}^{\tau_0}\Lambda_1 = \begin{cases}
        \Lambda_1 & \mathrm{if} \ \tau_0=1\\
        \mathbb{K} \Lambda_1=\overline{\Lambda_1} & \mathrm{if} \ \tau_0=-1.
    \end{cases}
\end{align}
The surjective homomorphism $\mathrm{GL}_n(\mathbb{C}) \rtimes \mathbb{Z}/2 \to \mathbb{Z}/2$ simply projects into the second coordinate.

If $V$ is a complex vector space, the magnetic general linear group $\mathrm{MGL}(V)$ of $V$
consists of the group of invertible endomorphisms of $V$ which are complex linear of complex antilinear. The projection $\mathrm{MGL}(V) \to \mathbb{Z}/2$ map the complex linear maps to zero and the complex antilinear maps to one. If one chooses a complex base base $V \cong \mathbb{C}^n$, then we obtain an isomorphism of magnetic groups
$\mathrm{MGL}(V) \cong \mathrm{MGL}_n(\mathbb{C})$.

A complex vector space $V$ becomes a magnetic representation of the magnetic group $(G, \phi)$ if it is endowed with a morphism of magnetic groups $\rho : G \to \mathrm{MGL}(V)$. A morphism between two magnetic representations $\rho_0: G \to \mathrm{MGL}(V_0)$ and $\rho_1: G \to \mathrm{MGL}(V_1)$ is a complex linear homomorphism $T: V_0 \to V_1$ such that $T \rho_0(g)= \rho_1(g)T$ for all $g \in G$. 

Any finite dimensional representation of a finite magnetic group is the direct sum of irreducible
magnetic representations. 
If $V$ is an irreducible magnetic representation of the magnetic group $(G, \phi)$, the algebra of $G$-invariant endomorphism of $V$ gives rise to the three-fold classification originally due to Wigner \cite{wigner}. If $V|_{G_0}$ denotes the complex representation restricted to $G_0= \mathrm{ker}(\phi)$, then we have the classification of magnetic irreducible representations by their type: 
\begin{align} \label{wigenr irreps}
    \mathrm{End}(V)^G = \begin{cases}
        \mathbb{R} & \mathrm{if} \ V|_{G_0} \ \mathrm{is \ irreducible},\\
        \mathbb{H} & \mathrm{if} \ V|_{G_0} \cong W \oplus W\\
        \mathbb{C} & \mathrm{otherwise}.
    \end{cases}
\end{align}
where $W$ denotes an irreducible complex $G_0$-representation.
Further properties of magnetic representations of finite magnetic groups can be read in the summary presented by Xicot\'encatl and the authors in \cite[\S 1]{serrano2025magneticequivariantktheory}, or in Wigner's book \cite{wigner}.

\subsection{Non-Abelian cohomology}
Let us recall the definition of the first cohomology pointed set $H^1(G,\mathcal{A})$ when $\mathcal{A}$ is a group which is not necessarily abelian.

Let $G$ be a group and $\mathcal{A}$ another group with a left action of $G$ by automorphisms. Denote by: 
\begin{equation}
    Z^1(G,\mathcal{A})    
\end{equation}
the set of all maps $f:G\to \mathcal{A}$ such that: 
\begin{equation}\label{1-nonabcond}
    f(gh)=f(g) \left(g\cdot f(h)\right)
\end{equation}
where $f(g)$ times $g\cdot f(h)$ uses the multiplication in $\mathcal{A}$ and $g\cdot f(h)$ is given by the action of $G$ on $\mathcal{A}$. Note that $\mathcal{A}$ acts on $Z^1(G,\mathcal{A})$ by conjugation:
\begin{equation}\label{nonabconju}
    (a\cdot f)(g):= a f(g)  (g\cdot(a^{-1}))
\end{equation}
where $a\in \mathcal{A}$, $f\in Z^1(G,\mathcal{A})$ and $g\in G$. In fact we have the equalities:
\begin{align*}
    (a\cdot f)(gh)= & a f(gh)  \left((gh)\cdot (a^{-1})\right) \\
    = & a f(g) (g \cdot f(h))  ((gh)\cdot (a^{-1})) \\
    = &  a f(g)( g\cdot a ) (g\cdot a^{-1}) (g \cdot f(h))  ((gh)\cdot (a^{-1})) \\
    = & (a\cdot f)(g) \left(g\cdot (a^{-1} f(h) (h\cdot(a^{-1})))\right) \\
    =& (a\cdot f)(g)  (g\cdot((a\cdot f)(h))),
\end{align*}
thus showing that $(a\cdot f) \in Z^1(G,\mathcal{A})$.

Define the first cohomology set:
\begin{equation}
    H^1(G,\mathcal{A}):= Z^1(G,\mathcal{A})/ \mathcal{A}
\end{equation}
as the set of equivalent classes of $Z^1(G,\mathcal{A})$ modulo the action of $\mathcal{A}$. The constant 1-cocycle
    $g\mapsto 1$
is called the base point of $H^1(G,\mathcal{A})$. Clearly, if $\mathcal{A}$ is abelian, then $H^1(G,\mathcal{A})$ is the usual first cohomology group of $G$ with coefficients in $\mathcal{A}$. Now, let us recall a fundamental result from \cite{Gille_Central_Simple_algebras} about exact sequences of pointed sets. A sequence of pointed sets: 
\begin{equation}
    (\mathcal{A},a)\overset{f}{\longrightarrow} (\mathcal{B},b) \overset{g}{\longrightarrow} (\mathcal{C},c)  
\end{equation}
is exact at $(\mathcal{B},b)$ if $g^{-1}(c)$, the kernel of $g$, is equal to the image of $f$.

From \cite[Prop. 2.7.1 \& 4.4.1]{Gille_Central_Simple_algebras} we quote:

\begin{proposition}\label{exactseqnonab}
    Let $G$ be a group and 
    \begin{equation}
        1\longrightarrow \mathcal{A} \longrightarrow \mathcal{B} \longrightarrow \mathcal{C} \longrightarrow 1
    \end{equation}
    an exact sequence of groups equipped with $G$-actions and with $G$-equivariant homomorphisms. Then
     there is an induced  exact sequence of pointed sets:
        \begin{equation}
            1 \rightarrow \mathcal{A}^G \rightarrow \mathcal{B}^G \rightarrow \mathcal{C}^G \rightarrow H^1(G,\mathcal{A}) \rightarrow H^1(G,\mathcal{B}) \rightarrow H^1(G,\mathcal{C}).
        \end{equation}
  Moreover, if $\mathcal{A}$ is abelian and it is contained in the center of $\mathcal{B}$, then the induced exact sequence of pointed sets extends one term to the right:
        \begin{equation}
            1 \rightarrow \mathcal{A}^G \rightarrow \mathcal{B}^G \rightarrow \mathcal{C}^G \rightarrow H^1(G,\mathcal{A}) \rightarrow H^1(G,\mathcal{B}) \rightarrow H^1(G,\mathcal{C}) \rightarrow H^2(G,\mathcal{A}).
            \end{equation}
\end{proposition}

 Suppose $\mathcal{A}$ is the direct limit of a $G$-directed set $(\mathcal{A}_k,\psi_k:\mathcal{A}_k\to \mathcal{A}_{k+1})_{k\in \mathbb{N}}$ such that for each $i,j\in \mathbb{N}$ there exist a map: 
\begin{equation} \label{product of direct limit}
    \Gamma_{ij}: \mathcal{A}_i\times \mathcal{A}_j \longrightarrow \mathcal{A}_{i+j} 
\end{equation}
such that:
\begin{itemize}
    \item $\Gamma_{ij}(g\cdot a, g\cdot b)=g\cdot\Gamma_{ij}(a,b)$ for all $g\in G$, $a\in \mathcal{A}_i$ and $b\in \mathcal{A}_j$.
    \item $\Gamma_{ij}(aa',bb')=\Gamma_{ij}(a,b)\cdot\Gamma_{ij}(a',b')$ for all $a,a'\in \mathcal{A}_i$ and $b,b'\in \mathcal{A}_j$.
    \item 
     $\Gamma_{i+1j}(\psi_i(a),b)=\psi_{i+j+1}(\Gamma_{ij}(a,b))=\Gamma_{ij+1}(a,\psi_j(b))$
    for all  $a\in \mathcal{A}_i$ and $b\in \mathcal{A}_j$.
\end{itemize}

In this case we can define a product on $H^1(G,\mathcal{A})$ as follows.
Define the product on the set of 1-cocycles by the formula:
\begin{align}\label{multofnonabcoh}
    Z^1(G,\mathcal{A}_i)\times Z^1(G,\mathcal{A}_j) & \longrightarrow Z^1(G,\mathcal{A}_{i+j}) \nonumber \\
    (f_0,f_1) & \longmapsto f_0 f_1  \hspace{1cm} (f_0 f_1)(a)=\Gamma_{ij}(f_0(a),f_1(a)),
\end{align}
and it can be shown that this product commutes with the action of $G$, thus inducing a well-defined product in cohomology:
\begin{align}
    H^1(G,\mathcal{A})\times H^1(G,\mathcal{A}) & \longrightarrow H^1(G,\mathcal{A}). 
\end{align}

\subsection{Magnetic equivariant projective representations.}

The magnetic projective linear group is the semi-direct product: 
\begin{align}
    \mathrm{MPGL}_n(\mathbb{C}) := \mathrm{PGL}_n(\mathbb{C}) \rtimes \mathbb{Z}/2
\end{align} 
where $\mathrm{PGL}_n(\mathbb{C})$ is the 
projective general linear group 
$\mathrm{MGL}_n(\mathbb{C})/\mathbb{C}^*$.
It
fits into the short exact sequence of groups:
\begin{align}
  1 \to  \mathbb{C}^*  \to \mathrm{MGL}_n(\mathbb{C}) \to \mathrm{MPGL}_n(\mathbb{C}) \to 1
\end{align}
where $\mathrm{MPGL}_n(\mathbb{C})$ acts on $ \mathbb{C}^*$
by complex conjugation via its image in the group $\mathbb{Z}/2$.

The inclusion of groups: 
\begin{align}
\mathrm{GL}_n(\mathbb{C}) \to \mathrm{GL}_{n+1}(\mathbb{C}), \ \ \ 
M \mapsto \left(\begin{smallmatrix}
    M & 0 \\
    0 & 1
\end{smallmatrix} \right)
\end{align}
allow us to define the direct limit of groups:
\begin{align}
    \mathrm{GL}_\infty(\mathbb{C}) = \lim_n \mathrm{GL}_n(\mathbb{C}),
\end{align}
and the groups $\mathrm{PGL}_\infty(\mathbb{C})$, $\mathrm{MGL}_\infty(\mathbb{C})$ and $\mathrm{MPGL}_\infty(\mathbb{C})$ are defined accordingly, all
of them fitting into the product in a direct limit presented in Eqn. \eqref{product of direct limit}.

A magnetic $(G,\phi)$-representation is nothing else than
a magnetic homomorphism $(G, \phi) \to \mathrm{MGL}_\infty(\mathbb{C})$
and a projective one is a homomorphism
$(G, \phi) \to \mathrm{MPGL}_\infty(\mathbb{C})$.
Because of the finite dimensionality property of the magnetic group $(G,\phi)$,
a magnetic $(G,\phi)$-representation factors through a finite dimensional one $(G, \phi) \to \mathrm{MGL}_n(\mathbb{C})$; the same
applies to magnetic projective representations.

Note that a magnetic representation: 
\begin{align}
    (G, \phi) \to \mathrm{MGL}_\infty(\mathbb{C}), \ \ g \mapsto (\Lambda_g, \phi(g))
\end{align}
can also be represented as a $\phi$-twisted homomorphism from $G$ to $\mathrm{GL}_\infty(\mathbb{C})$.
The product in the group $g \cdot h = gh$ transforms into the equation:
\begin{align}
    (\Lambda_g,\phi(g))(\Lambda_h, \phi(h))=(\Lambda_{(gh)}, \phi(gh))
\end{align}
which implies the equation:
\begin{align}
    \Lambda_g \ {\mathbb{K}}^{\phi(g)}{\Lambda_h} = \Lambda_{gh}.
\end{align}
Hence a magnetic representation induces an assignment $g \mapsto \Lambda_g$ such that: 
\begin{align}\label{twistedhom}
    \Lambda_g \ {\mathbb{K}}^{\phi(g)}{\Lambda_h} \  {\Lambda_{gh}}^{-1}=1,
\end{align}
which is precisely the definition of a $\phi$-twisted homomorphism from $G$ to $\mathrm{GL}_\infty(\mathbb{C})$.  

We see that the condition of $\phi$-twisted homomorphism \ref{twistedhom} is nothing but the 1-cocycle condition \ref{1-nonabcond} for $A=\mathrm{GL}_\infty(\mathbb{C}_\phi)$, so we have: 
\begin{align}
    Z^1(G, \mathrm{GL}_\infty(\mathbb{C}_\phi))= & \,\{\phi-\text{twisted homomorphisms from } G \text{ to } \mathrm{GL}_\infty(\mathbb{C}_\phi)\} \\
    Z^1(G, \mathrm{PGL}_\infty(\mathbb{C}_\phi))= & \, \{\phi-\text{twisted homomorphisms from } G \text{ to } \mathrm{PGL}_\infty(\mathbb{C}_\phi)\}
\end{align}
and therefore we have the following isomorphisms of sets:
\begin{align}
    Z^1(G, \mathrm{GL}_\infty(\mathbb{C}_\phi)) &\cong 
    \mathrm{Hom}_{Mag}(G,\mathrm{MGL}_\infty(\mathbb{C})),\\
    Z^1(G, \mathrm{PGL}_\infty(\mathbb{C}_\phi)) &\cong 
    \mathrm{Hom}_{Mag}(G,\mathrm{MPGL}_\infty(\mathbb{C})).
\end{align}

We say that two $(G,\phi)$-representations are equivalent, or conjugate, if there is a complex change of basis from one to the other. That is, for magnetic representations $\rho, \rho': (G,\phi) \to \mathrm{MGL}_n(\mathbb{C})$ , they are equivalent if there exists $M \in \mathrm{GL}_n(\mathbb{C})$ such that:
\begin{align}
\rho'= (M,1)\ \rho\  (M^{-1},1).
\end{align}
For projective representations the same definition of 
equivalence holds. 

In terms of $\phi$-twisted homomorphisms the conjugation equation 
looks like:
\begin{align}
    M \  \Lambda_g \ {\mathbb{K}}^{\phi(g)}M^{-1} = \Lambda_g',
\end{align}
which is nothing but condition \ref{nonabconju}, and therefore we have bijections of pointed sets:
\begin{align}
  H^1(G,\mathrm{GL}_\infty(\mathbb{C_\phi) }) & \cong   \mathrm{Hom}_{Mag}(G,\mathrm{MGL}_\infty(\mathbb{C}))/_{equiv}\\
  H^1(G,\mathrm{PGL}_\infty(\mathbb{C_\phi) })&\cong  \mathrm{Hom}_{Mag}(G,\mathrm{MPGL}_\infty(\mathbb{C}))/_{equiv}.
\end{align}
where the right-hand side denote the set of equivalence classes of magnetic (projective) representations and the base point of the right-hand side is the representation $g\mapsto \mathbb{K}^{\phi(g)}$.

Note that the set $H^1(G,\mathrm{GL}_\infty(\mathbb{C_\phi) })$
can be made into an abelian group by the construction \ref{multofnonabcoh} with $\Gamma_{ij}$ the tensor product of representations.
This is the map:
\begin{align}
    \mathrm{Hom}_{Mag}(G,\mathrm{MGL}_p(\mathbb{C})) \times \mathrm{Hom}_{Mag}(G,\mathrm{MGL}_q(\mathbb{C})) & \to \mathrm{Hom}_{Mag}(G,\mathrm{MGL}_{pq}(\mathbb{C}))\\
    \ (\rho,\rho') & \mapsto \rho \otimes \rho'
\end{align}
which is compatible with the equivalence classes of representations. 
Projective representations can also be tensored and thus 
$H^1(G,\mathrm{PGL}_\infty(\mathbb{C_\phi) })$ is also an abelian group and 
the canonical map:
\begin{align}
    H^1(G,\mathrm{GL}_\infty(\mathbb{C_\phi) }) \to H^1(G,\mathrm{PGL}_\infty(\mathbb{C_\phi) })
\end{align}
is a homomorphism of groups.

Now, by Prop. \ref{exactseqnonab} we know that the short exact sequence of groups with $G$-action:
\begin{align}
  1 \to  \mathbb{C^*_\phi}\to \mathrm{GL}_\infty(\mathbb{C}_\phi) \to \mathrm{PGL}_\infty(\mathbb{C}_\phi) \to 1,
\end{align}
where $G$ acts through $\phi$ by complex conjugation on the three groups,
induces the long exact sequence of groups:
\begin{align} \nonumber
    (\mathbb{C}^*_\phi)^G \to \mathrm{GL}_\infty(\mathbb{C}_\phi)^G \to \mathrm{PGL}_\infty(\mathbb{C}_\phi)^G \to H^1(G, \mathbb{C}^*_\phi)
    \to&\\
    H^1(G, \mathrm{GL}_\infty(\mathbb{C}_\phi)) \to&  H^1(G, \mathrm{PGL}_\infty(\mathbb{C}_\phi)) \to H^2(G, \mathbb{C}^*_\phi).
\end{align}

Here the connection homomorphism $ H^1(G, \mathrm{PGL}_\infty(\mathbb{C}_\phi)) \to H^2(G, \mathbb{C}^*_\phi)$
can be understood as the assignment: 
\begin{align}
\mathrm{Hom}_{Mag}(G,\mathrm{MPGL}_\infty(\mathbb{C}))/_{equiv} & \to H^2(G, \mathbb{C}^*_\phi)\\
\left[  \psi : G \to \mathrm{MPGL}_\infty(\mathbb{C}) \right]
 &  \mapsto [\widetilde{G} := \psi^* \mathrm{MGL}_\infty(\mathbb{C})]
\end{align}
where $\widetilde{G}$ is the pullback in the following diagram:
\begin{align}
    \xymatrix{ 1 \ar[r] &
\mathbb{C}^* \ar@{-}[d] \ar[r] & \widetilde{G} \ar[d]\ar[r] & G \ar[d]^{\psi}  \ar[r] & 1\\
1 \ar[r] & \mathbb{C}^* \ar[r] & \mathrm{MGL}_\infty(\mathbb{C}) \ar[r] & \mathrm{MPGL}_\infty(\mathbb{C})  \ar[r] & 1,
}
\end{align}
and $[\widetilde{G}] \in H^2(G, \mathbb{C}^*_\phi)$ is the equivalence class of the extension $\widetilde{G}$ of $G$
by $\mathbb{C}^*$.

Generalizing to the equivariant setup  \cite[Theo. 4.4.5]{Gille_Central_Simple_algebras}, we claim the following result:
\begin{proposition} \label{prop twisted reps = cocycles}
    The canonical map:
    \begin{align}
        H^1(G, \mathrm{PGL}_\infty(\mathbb{C}^*_\phi)) {\to} H^2(G, \mathbb{C}^*_\phi)
    \end{align}
is surjective. Therefore we have an isomorphism:
\begin{align}
    \frac{H^1(G, \mathrm{PGL}_\infty(\mathbb{C}^*_\phi))}
    {{\mathrm{image}}(H^1(G, \mathrm{GL}_\infty(\mathbb{C}^*_\phi)))} \cong H^2(G, \mathbb{C}^*_\phi)
\end{align}
\end{proposition}

\begin{proof}
By the classification of extension of groups we know that the group 
$H^2(G, \mathbb{C}^*_\phi)$ classifies extensions of $G$ by $\mathbb{C}^*$
where $G$ acts on $\mathbb{C}^*$ via $\phi$ by complex conjugation. Therefore an element in $H^2(G, \mathbb{C}^*_\phi)$ defines an extension of groups:
\begin{align} \label{extension of groups tilde G}
    1 \to \mathbb{C}^* \to \widetilde{G} \overset{\pi}{\to} G \to 1
\end{align}
where the magnetic structure is given by the homomorphism $\widetilde{\phi} := \pi \circ \phi$.
The core $\widetilde{G}_0= \mathrm{ker}(\widetilde{\phi}:\widetilde{G} \to \mathbb{Z}/2)$ of the magnetic group $\widetilde{G}$ is a central extension $\mathbb{C}^* \to \widetilde{G}_0 \to G_0$
and by the Peter-Weyl theorem it has an irreducible representation $V$
where $\mathbb{C}^*$ acts by multiplication. According to Wigner \cite{wigner} this representation
can be lifted to a magnetic representation $W$
of $\widetilde{G}$ and therefore it defines an element
in $ \mathrm{Hom}_{Mag}(\widetilde{G},\mathrm{MGL}_\infty(\mathbb{C}))$
whose projectivization defines an element in 
$ \mathrm{Hom}_{Mag}(G,\mathrm{MPGL}_\infty(\mathbb{C}))$.
The image of this element in  $H^1(G, \mathrm{PGL}_\infty(\mathbb{C}^*_\phi))$ is precisely the class of the extension $\widetilde{G}$. This shows the surjectivity
of the connection homomorphism.

\end{proof}
We can see the previous argument at the level of cocycles.
    Recall that
     $ H^*(G, \mathbb{C}^*_\phi)$ is the cohomology of the
 cochain complex:
\begin{align}
     C^n(G, \mathbb{C}^*_\phi) := \mathrm{Maps}(G^n, \mathbb{C}^*)
\end{align}
with boundary map:
\begin{align}
     \delta: C^n(G, \mathbb{C}^*_\phi) &\to  C^{n+1}(G, \mathbb{C}^*_\phi) \\
     \delta f(g_0,...,g_n)&= {}^{\phi(g_0)}f(g_1,...,g_n)\prod_{j=1}^{n} f(g_0, ..., g_{j-1}{} g_{j}, ... g_n)^{(-1)^j}f(g_0, ...,g_{n-1})^{(-1)^{n+1}}.
\end{align}
For a 2-cocycle $\tau \in Z^2(G, \mathbb{C}^*_\phi)$ we have the equation:
\begin{align} \label{2-cocylce eqn}
    {}^{\phi(g)}\tau(h,k) \tau(g, hk) = \tau (gh, k) \tau(g,h)
\end{align}
and we may define the group $\widetilde{G}$ as the set $\mathbb{C}^* \times G$ with product structure given by the formula:
\begin{align}
    (\lambda_0,g_0)(\lambda_1,g_1) := (\tau(g_0,g_1)\lambda_0 \ {}^{\phi(g_0)}\lambda_1 , g_0g_1).
\end{align}
The magnetic group $\widetilde{G}$ is a compact Lie group and
therefore the existence of a representation of $\widetilde{G}$ where $\mathbb{C}^*$ acts by scalar multiplication is ensured by the Peter-Weyl theorem.

A magnetic projective magnetic representation of the group $G$ can be encoded in matrices
$\Lambda_g$ for $g \in G$ such that: 
\begin{align}
    \Lambda_g \ \mathbb{K}^{\phi(g)} \Lambda_h = \tau(g,h) \Lambda_{gh}
\end{align}
where $\tau(g,h) \in \mathbb{C}^*$ keeps track of the discrepancy satisfying Eqn. \eqref{2-cocylce eqn}.
If the cocycle $\tau$ is trivializable, i.e.: 
\begin{align}
    \tau(g,h) = \delta \alpha (g,h) = {}^{\phi(g)}\alpha(h) \alpha(gh)^{-1}\alpha(g),
\end{align}
then the projective magnetic representation can be lifted to a 
magnetic representation of $G$ with assignment $g \mapsto \alpha(g)^{-1}\Lambda_g \mathbb{K}^{\phi(g)}$.

\subsection{Magnetic equivariant graded Brauer group}
Let $(G, \phi)$ be a magnetic group and $G_0 := \mathrm{ker}(\phi)$. We say that a graded algebra $A=A_0 \oplus A_1$ over $\mathbb{C}$ is equivariant with respect to the magnetic group $(G , \phi)$, if there is a homomorphism:
\begin{align}
    \tau : G \to \mathrm{MAut}_{\mathrm{gr}}(A), \ \ \ \ 
\end{align}
to the magnetic group of automorphisms $\mathrm{MAut}_{\mathrm{gr}}(A)$ consisting of $\mathbb{C}$-linear and $\mathbb{C}$-antilinear graded algebra automorphisms of $A$.

 In particular we have that for
$\lambda \in \mathbb{C}$ and $a_1,a_2 \in A$:
\begin{align}
    \tau(g)(\lambda a_1 \cdot a_2) = (\mathbb{K}^{\phi(g)}\lambda) \tau(g)(a_1) \cdot \tau(g)(a_2),
\end{align}
where: 
\begin{align}
    \mathbb{K}^{\phi(g)}\lambda =\begin{cases}
        \mathbb{K} \lambda = \overline{\lambda} & \mathrm{if} \ \phi(g)=1 \\
        \lambda & \mathrm{if} \ \phi(g)=0.
    \end{cases}
\end{align}
In most cases the use of $\tau$ will be suppressed and we will write
 $ga=\tau(g)(a)$.

For $B$ another $(G, \phi)$-equivariant graded algebra $\tau' : G \to \mathrm{MAut}_{\mathrm{gr}}(B)$, the graded tensor product $A \hat{\otimes} B$
is $(G,\phi)$-equivariant with the action $g(a \hat{\otimes} b)=(ga) \hat{\otimes} (gb)$ for $g \in G$. The canonical isomorphisms:
\begin{align}
    (A \hat{\otimes}B) \hat{\otimes}C \to A \hat{\otimes}(B \hat{\otimes}C )
    \ \ \ \mathrm{and} \ \ \ A \hat{\otimes}B \to  B \hat{\otimes}A, \ a\hat{\otimes b} \mapsto (-1)^{\partial a \partial b} b \hat{\otimes}
a\end{align}
are also $(G,\phi)$-equivariant and the map associated to the $(G,\phi)$ action on $A \hat{\otimes} B$ is denoted $\tau \hat{\otimes}\tau' : G \to
\mathrm{MAut}_{\mathrm{gr}} (A \hat{\otimes}B)$. In fact, 
    the category of $(G,\phi)$-equivariant graded algebras together with the graded tensor product $\hat{\otimes}$ and the trivial $(G,\phi)$-equivariant graded algebra $\mathbb{C}$ is a symmetric monoidal category.

We may consider the category of magnetic $(G,\phi)$-equivariant central and simple graded algebras  MECSGA$_{(G,\phi)}(\mathbb{C})$ and we may define an equivalence relation of similarity in this context as follows.
Consider direct sums of magnetic $(G,\phi)$-representations:
\begin{align}
    (V, \rho) = (V_0, \rho_0) \oplus (V_1, \rho_1)
\end{align}
where $\rho_i: G \to \mathrm{MGL}(V_i)$ are homomorphisms to the magnetic
general linear group of the complex vector spaces $V_i$.
Consider the graded algebra of endomorphisms:
\begin{align} \label{End(V0,V1)}
    \mathrm{End}(V_0,V_1)
\end{align}
and endow it with the $(G,\phi)$ action given by conjugation by $\rho$. Namely, $F \in \mathrm{End}(V_0,V_1)$ is mapped to $\rho(g) F \rho(g)^{-1}$
for all $g \in G$.

Two MECSGA$_{(G,\phi)}(\mathbb{C})$s $A$ and $B$ are similar, denoted as $A  \sim B$, if there are magnetic $(G,\phi)$-representations $V_i$,$W_i$ such that there is an isomorphism:
\begin{align} \label{eqn tensor reps}
    A \hat{\otimes} \mathrm{End}(V_0,V_1) \cong B \hat{\otimes} \mathrm{End}(W_0,W_1)
\end{align}
of graded algebras which is moreover $G$-equivariant.

Similarity is an equivalence relation in MECSGA$_{(G,\phi)}(\mathbb{C})$, and for this to hold, one has to check that the canonical isomorphism:
\begin{align} \label{Eqn tensor of algebras}
    \mathrm{End}(V_0,V_1) \hat{\otimes} \mathrm{End}(W_0,W_1) \cong 
    \mathrm{End}((V\hat{\otimes}W)_0,(V\hat{\otimes} W)_1)
\end{align}
is $G$-equivariant.

The equivalence classes of similar MECSGA$_{(G,\phi)}(\mathbb{C})$s
is called the {\it Magnetic $(G,\phi)$-equivariant Brauer group} and will be denoted:
\begin{align}
    \mathrm{GrBr}_{(G,\phi)}(\mathbb{C}).
\end{align}

Note that in the case $G=\mathbb{Z}/2$ and $\phi=\mathrm{id}$, we recover 
the Graded Brauer group $\mathrm{GrBr}_{(\mathbb{Z}/2,\mathrm{id})}(\mathbb{C})$ studied in \S \ref{subsection algebras antilinear}.

The graded tensor product makes the magnetic $(G,\phi)$-equivariant Brauer group $\mathrm{GrBr}_{(G,\phi)}(\mathbb{C})$ into an abelian group.
Let us understand its structure.

\begin{proposition} \label{proposition extensions}
    The  graded magnetic $(G,\phi)$-equivariant Brauer group fits into the non-split short exact sequence:
    \begin{align} \label{short exact sequence GrBR(G)}
     0 \to  \mathrm{GrBr}_{(G,\phi)}(\mathbb{C})' \to \mathrm{GrBr}_{(G,\phi)}(\mathbb{C}) \to \mathrm{GrBr}(\mathbb{C})\cong \mathbb{Z}/2  \to 0
    \end{align}
    where the second map forgets the $(G,\phi)$-equivariant structure of the algebras, and its kernel fits into the non-generally split short exact sequence:
    \begin{align} \label{short exact sequence GrBR(G)'}
   0\to     H^2(G, \mathbb{C}^*_\phi) \to \mathrm{GrBr}_{(G,\phi)}(\mathbb{C})' \to \mathrm{Hom}(G,\mathbb{Z}/2) \to 0
    \end{align}
where $G$ acts by complex conjugation on  the coefficients $\mathbb{C}^*_\phi$ via 
the homomorphism $\phi.$
\end{proposition}

\begin{proof} A very important case when the exact sequences are not split is when the magnetic group
is $(\mathbb{Z}/2, \mathrm{id})$. In this case we know that $\mathrm{GrBr}_{(\mathbb{Z}/2, \mathrm{id})}(\mathbb{C})= \mathbb{Z}/8$
and 
\begin{align}
    H^2(\mathbb{Z}/2, \mathbb{C}^*_{\mathrm{id}}) = \mathbb{Z}/2, \ \ \mathrm{Hom}(\mathbb{Z}/2, \mathbb{Z}/2)= \mathbb{Z}/2, \ \ \ 
    \mathrm{GrBr}(\mathbb{C}) = \mathbb{Z}/2.
\end{align}
Therefore $\mathrm{GrBr}_{(\mathbb{Z}/2, \mathrm{id})}(\mathbb{C})'= \mathbb{Z}/4$ and none of the sequences split.

Now consider the natural homomorphism
of magnetic groups $(G,\phi) \to (\mathbb{Z}/2, \mathrm{id})$ that $\phi$ defines, take the induced homomorphism, and note
that it is compatible with the forgetful homomorphism making the following diagram commutative:
\begin{align}
    \xymatrix{
    \mathrm{GrBr}_{(G,\phi)}(\mathbb{C})' \ar[r] & \mathrm{GrBr}_{(G,\phi)}(\mathbb{C}) \ar[r] & \mathrm{GrBr}(\mathbb{C}) \cong \mathbb{Z}/2\\
   \mathrm{GrBr}_{(\mathbb{Z}/2,\mathrm{id})}(\mathbb{C})' \cong \mathbb{Z}/4 \ar[r] \ar[u]&   \mathrm{GrBr}_{(\mathbb{Z}/2,\mathrm{id})}(\mathbb{C}) \cong \mathbb{Z}/8.  \ar[u] \ar@{->>}[ur] &
    }
\end{align}

Therefore the forgetful homomorphism is surjective
and the algebra $\mathbb{C}[e]/\langle e^2=1 \rangle$ with $G$ action defined by the automorphism $\mathrm{Ad}_{\mathbb{K}^{\phi(g)}}$
belongs to $\mathrm{GrBr}_{(G,\phi)}(\mathbb{C})$
and maps to the generator in $\mathrm{GrBr}(\mathbb{C})$.

Consider now the kernel 
$\mathrm{GrBr}_{(G,\phi)}(\mathbb{C})'$
of the forgetful map  and take an algebra
$A$ in this subgroup. The algebra $A$ is isomorphic to a graded matrix algebra $M_{k|l}(\mathbb{C})$ for $k,l \geq 0$, and $G$ acts by graded automorphisms.
By the Skolem-N\"other Theorem \cite[Thm. 2.7.2]{Gille_Central_Simple_algebras} the group $G$ acts by inner automorphisms in $M_{k|l}(\mathbb{C})$, and therefore for every $g$
in $G$ there are matrices $M_g$ 
such that the $G$ action is given by the inner automorphisms:
\begin{align}
    G \to \mathrm{MAut}_{\mathrm{gr}}(M_{k|l}(\mathbb{C})), \ \ \ g \mapsto \mathrm{Ad}_{M_g \mathbb{K}^{\phi(g)}}.
\end{align}
Since the graded automorphism acts in the degree zero part of the graded algebra $M_{k|l}(\mathbb{C})_0$,
then the matrices $M_g$ must be homogeneous.
Therefore the matrices $M_g$ have a degree $\partial (M_g)$,
where $\partial (M_g)=0$ if $M_g$ is of the form $\left(\begin{smallmatrix}
            * & 0 \\ 0 & *
        \end{smallmatrix}\right)$ and 
        $\partial (M_g)=1$ if $M_g$ is of the form $\left(\begin{smallmatrix}
            0 & * \\ * & 0
        \end{smallmatrix}\right)$. This degree induces 
an assignment:
\begin{align}
    \psi_A:  G \to \mathbb{Z}/2, \ \ g \mapsto \partial(M_g)
\end{align}    
that becomes a homomorphism of groups because $M_g \mathbb{K}^{\phi(g)}M_h$ and 
$M_{gh}$ only differ by a scalar. 

Note that the endomorphism algebra $\mathrm{End}(V_0,V_1)$, for  $V_0$ and $V_1$ magnetic $(G, \phi)$-representations, has a canonical $G$-action given by conjugation. Since the action of $G$ on $V_0$ and $V_1$ is of degree zero, then the associated degree to each element $g \in G$ is zero. Hence we have that the associated homomorphism $\psi_{\mathrm{End}(V_0,V_1)}: G \to \mathbb{Z}/2$ is trivial.
Therefore we have an isomorphism:
\begin{align}
    \psi_A = \psi_{\mathrm{End}(V_0,V_1)},
\end{align}
which implies that the similarity class of $A$
can be mapped to the homomorphism $\psi_A \in \mathrm{Hom}(G, \mathbb{Z}/2)$, and the graded tensor product structure in
$\mathrm{GrBr}_{(G,\phi)}(\mathbb{C})'$ induces the desired homomorphism:
\begin{align} \label{homomorphisms_GrBr'-to-Hom(G,Z2)}
    \mathrm{GrBr}_{(G,\phi)}(\mathbb{C})' \to \mathrm{Hom}(G, \mathbb{Z}/2), \ \  A \mapsto \psi_A.
\end{align}

This homomorphism is surjective because of the following construction. Take any homomorphism $\psi \in \mathrm{Hom}(G, \mathbb{Z}/2)$ and take the algebra $A= M_{1|1}(\mathbb{C})$
where the group $G$ acts by graded automorphisms in the following form:
\begin{align} \label{Algebra for a homomorphism psi}
    \tau_\psi(g) =
    \begin{cases}
        \mathrm{Ad}_{\mathbb{K}^{\phi(g)}} & \mathrm{if}\ \   \psi(g)=0 \\
        \mathrm{Ad}_{\left(\begin{smallmatrix}
            0 & 1 \\ 1 & 0
        \end{smallmatrix}\right)   \mathbb{K}^{\phi(g)}} & \mathrm{if} \  \ \psi(g)=1 .
    \end{cases}
\end{align}

We have that the similarity class of  
$A$ with $G$-action defined by $\tau_\psi$
defines an element in $\mathrm{GrBr}_{(G,\phi)}(\mathbb{C})'$,
and by construction $\psi_A=\psi$. Therefore the assignment 
of Eqn. \eqref{homomorphisms_GrBr'-to-Hom(G,Z2)} is a surjective homomorphism of groups.

Now let us assume that $\psi_A$ is the trivial homomorphism for $A = \mathrm{End}(V)$ with $V=V_0\oplus V_1$. Then $A \in \mathrm{GrBr}_{(G,\phi)}(\mathbb{C})''$ where we define:
\begin{align}
    \mathrm{GrBr}_{(G,\phi)}(\mathbb{C})'' = \mathrm{ker} \left( \mathrm{GrBr}_{(G,\phi)}(\mathbb{C})' \to \mathrm{Hom}(G, \mathbb{Z}/2) \right).
\end{align}

This implies that the endomorphisms $M_g$ are all of degree zero,  
and therefore this action of $G$ on $A$ defines a magnetic projective representation on $V_0 \oplus V_1$. Note that $\mathrm{End}(V_0,V_1)$ and $\mathrm{End}(V_0\oplus V_1,0)$ are similar because we can tensor each algebra by the trivial representation $\mathrm{End}(\mathbb{C},\mathbb{C})$, so in $\mathrm{GrBr}_{(G,\phi)}(\mathbb{C})''$ we can forget about the graduation of the vector spaces. Hence we can construct a map:
\begin{align}
      \mathrm{GrBr}_{(G,\phi)}(\mathbb{C})'' &\to \tfrac{H^1(G, \mathrm{PGL}_\infty(\mathbb{C}^*_\phi))}
    {{\mathrm{image}}(H^1(G, \mathrm{GL}_\infty(\mathbb{C}^*_\phi)))} \\
    \mathrm{End}(V) & \mapsto [ g \mapsto M_g]
\end{align}
which is well-defined because of Eqn. \eqref{eqn tensor reps}, it is an injective homomorphism of groups because of Eqn. \ref{Eqn tensor of algebras}, and  it is surjective because for any magnetic projective representation $W$, the $G$ action on the algebra $\mathrm{End}(W)$ permits to recover $W$.

In view of Prop. \ref{prop twisted reps = cocycles} we therefore have that: 
\begin{align} \label{GrBr'' = H2}
      \mathrm{GrBr}_{(G,\phi)}(\mathbb{C})'' &\cong H^2(G,\mathbb{C}^*_\phi)
\end{align}
and we obtain the short exact sequence of Eqn. \eqref{short exact sequence GrBR(G)'}:
\begin{align} 
   0\to     H^2(G, \mathbb{C}^*_\phi) \to \mathrm{GrBr}_{(G,\phi)}(\mathbb{C})' \to \mathrm{Hom}(G,\mathbb{Z}/2) \to 0.
    \end{align}
 Here an element $ \lambda \in  H^2(G, \mathbb{C}^*_\phi)$ defines an extension of magnetic groups $\widetilde{G}_\lambda$ as in Eqn. \eqref{extension of groups tilde G}. If $V_\lambda$ is a magnetic representation of $\widetilde{G}_\lambda$ where 
 $\mathbb{C}^* = \mathrm{ker}(\widetilde{G} \to G)$ acts by multiplication of scalars and $\widetilde{g} \mapsto M_{\widetilde{g}} \in \mathrm{End}(V_\lambda)$,
 then the element $\lambda$ in
 $ H^2(G, \mathbb{C}^*_\phi)$ is mapped to the algebra $\mathrm{End}(V_\lambda)$ 
 with automorphism: 
 \begin{align} \label{inner automorpshism twisted group}
 \tau_\lambda(g)=\mathrm{Ad}_{M_{\widetilde{g}} \mathbb{K}^{\phi(g)}}
 \end{align}
 where $g \in G$ and $\widetilde{g}$ is any lift of $g$ in $\widetilde{G}_\lambda$. Therefore we have the assignment: 
 \begin{align} \label{H2 to GrBr''}
     H^2(G, \mathbb{C}^*_\phi) &\overset{\cong}{\to } \mathrm{GrBr}_{(G,\phi)}(\mathbb{C})'' \\
     \lambda &\mapsto (\mathrm{End}(V_\lambda), \tau_\lambda) \nonumber
 \end{align}
 which realizes the isomorphism since both algebras
 $\mathrm{End}(V_{\lambda_0}) \hat{\otimes} \mathrm{End}(V_{\lambda_0}) $  and 
 $\mathrm{End}(V_{\lambda_0\lambda_1}) $ are similar.
\end{proof}

The 2-cocycle associated to the extension of Eqn. \eqref{short exact sequence GrBR(G)'} is given by a set theoretical map:
\begin{align}
   \iota \circ \rho : \mathrm{Hom}(G,\mathbb{Z}/2) \times \mathrm{Hom}(G,\mathbb{Z}/2) \to 
    H^2(G, \mathbb{C}^*_\phi)
\end{align}
which is defined by the composition of the 2-cocycle:
\begin{align}
 \rho:  \left( \mathrm{Hom}(G,\mathbb{Z}/2) \times \mathrm{Hom}(G,\mathbb{Z}/2) \right){\to} H^2(G, \mathbb{Z}/2) 
\end{align}
with the homomorphism:
\begin{align}
   \iota:  H^2(G, \mathbb{Z}/2) \to
    H^2(G, \mathbb{C}^*_\phi)
\end{align}
induced by the 
 homomorphism in coefficients $\mathbb{Z}/2 \to \mathbb{C}^*_\phi$.

\begin{proposition} \label{proposition 2-cocycle rho}
The 2-cocycle:  
\begin{align}
    \rho \in Z^2\left(\mathrm{Hom}(G,\mathbb{Z}/2), H^2(G, \mathbb{Z}/2)\right)
\end{align}
of the extension of Eqn. \eqref{short exact sequence GrBR(G)'} is defined by the formula:
\begin{align}\label{formula for cocycle rho}
  \rho ( \psi_0 , \psi_1 ) (g,h) : = (-1)^{\psi_0(g) \psi_0(h)+\psi_0(g) \psi_1(h)+\psi_1(g) \psi_1(h)} 
\end{align}
where $\psi_0,\psi_1 \in \mathrm{Hom}(G,\mathbb{Z}/2)$, $g,h \in G$,
and 
the $\mathbb{Z}/2$ in $\mathrm{Hom}(G,\mathbb{Z}/2)$ is taken with the additive group structure of $\{0,1\}$, while the $\mathbb{Z}/2$  in $H^2(G, \mathbb{Z}/2)$ is taken with the multiplicative group structure of $\{1,-1\}$.
\end{proposition}

\begin{proof}
This explicit choice of 2-cocycle follows from the construction done in  Eqn. \eqref{homomorphisms_GrBr'-to-Hom(G,Z2)} of a prescribed  MECSGA$_{(G,\phi)}(\mathbb{C})$
for each given homomorphisms $\psi \in \mathrm{Hom}(G, \mathbb{Z}/2)$. 
Let us elaborate.

Define the set theoretical section of the homomorphism of Eqn. \eqref{homomorphisms_GrBr'-to-Hom(G,Z2)}
by the assignment:
\begin{align}
    \sigma : \mathrm{Hom}(G,\mathbb{Z}/2) &\to \mathrm{GrBr}_{(G,\phi)}(\mathbb{C})'\\
    \psi &\mapsto (C^{2,0}_\mathbb{R}\underset{\mathbb{R}}{\otimes} \mathbb{C},\tau_\psi)
    \label{sigma of psi}
\end{align}
where $C^{2,0}_\mathbb{R}$ is the real Clifford algebra and the $G$ action is defined by the equation:
\begin{align} \label{Algebra for a homomorphism psi2} 
    \tau_\psi(g) =
        \mathrm{Ad}_{{e_1}^{\psi(g)}   \mathbb{K}^{\phi(g)}}
\end{align}
with $e_1^2=1$ and the adjoint action conjugates by $e_1^{\psi(g)}$. Here we are using the 
results presented in Eqn. 
\eqref{C{2,0} to complex} of the complexification of the real Clifford algebra $C^{2,0}_\mathbb{R}$.

Now take two homomorphism $\psi_0, \psi_1 \in \mathrm{Hom}(G, \mathbb{Z}/2)$ and denote their product by  $\psi_{01}:=\psi_0+\psi_1$.
The failure of $\sigma$ to be a homomorphism 
will provide us with the 2-cocycle condition that we are looking for.
Consider the product of MECSGA$_{(G,\phi)}(\mathbb{C})$'s:
\begin{align} \label{cocycle in rho}
    \sigma(\psi_0) \hat{\otimes} \sigma(\psi_1) \hat{\otimes} \sigma(\psi_{01})^{-1} 
\end{align}
and note that this product lies in 
$\mathrm{GrBr}_{(G,\phi)}(\mathbb{C})''$. 
The projective representation that the $G$-action defines will return the desired cocycle.

The inverse of the algebra $\sigma(\psi_{01})$ appearing in Eqn. \eqref{cocycle in rho} 
is given by the opposite algebra of the complexification of the Clifford algebra $C^{2,0}_\mathbb{R}$. Therefore we may take:
\begin{align}
    \sigma(\psi_{01})^{-1} = ( (C^{0,2}_\mathbb{R}\underset{\mathbb{R}}{\otimes} \mathbb{C},\tau^{\mathrm{op}}_{\psi_{01}})
\end{align}
where $\tau^{\mathrm{op}}_{\psi_{01}}(g)=\mathrm{Ad}_{e_5 \mathbb{K}^{\phi(g)}}$
whenever $\psi_{01}(g)=1$ and $e_5^2=-1$.

Fixing notation we will have the equations:
\begin{align}
    C^{4,2}_\mathbb{R} \cong C^{2,0}_\mathbb{R} \hat{\otimes}
    C^{2,0}_\mathbb{R} \hat{\otimes}
    C^{0,2}_\mathbb{R}, \ \ \ \ e_1^2=e_2^2=e_3^2=e_4^2=1, \ \  \ \ e_5^2=e_6^2=-1
\end{align}
with $\tau_{\psi_0}(g) = \mathrm{Ad}_{e_1 \mathbb{K}^{\phi(g)}}$ whenever $\psi_{0}(g)=1$, and $\tau_{\psi_1}(g)=\mathrm{Ad}_{e_3 \mathbb{K}^{\phi(g)}}$
whenever $\psi_{1}(g)=1$.

The algebra of Eqn. \eqref{cocycle in rho}  is therefore:
\begin{align}
\sigma(\psi_0) \hat{\otimes}\sigma(\psi_1) \hat{\otimes}\sigma(\psi_{01})^{-1} =  & 
(C^{2,0}_\mathbb{R}\underset{\mathbb{R}}{\otimes} \mathbb{C},\tau_{\psi_{1}})
\hat{\otimes}
(C^{2,0}_\mathbb{R}\underset{\mathbb{R}}{\otimes} \mathbb{C},\tau_{\psi_{1}})
\hat{\otimes}
(C^{0,2}_\mathbb{R}\underset{\mathbb{R}}{\otimes} \mathbb{C},\tau^{\mathrm{op}}_{\psi_{01}})\\
 \ \cong &   \left(C^{4,2}_\mathbb{R} \underset{\mathbb{R}}{\otimes} \mathbb{C}, \tau_{\psi_0} \hat{\otimes} \tau_{\psi_1} \hat{\otimes} \tau^{\mathrm{op}}_{\psi_{01}}\right).
\end{align}

Denoting:
\begin{align}
    \tau : = \tau_{\psi_0} \hat{\otimes} \tau_{\psi_1} \hat{\otimes} \tau^{\mathrm{op}}_{\psi_{01}}
\end{align}
we see that: 
\begin{align}
    \tau(g) = \mathrm{Ad}_{e_1^{\psi_0(g)}e_3^{\psi_1(g)}e_5^{\psi_{01}(g)} \mathbb{K}^{\phi(g)}}
\end{align}
and therefore 
the 2-cocycle $\rho(\psi_0,\psi_1)$ is defined by the equation:
\begin{align}
    \left( e_1^{\psi_0(g)}e_3^{\psi_1(g)}e_5^{\psi_{01}(g)} \right) \left( e_1^{\psi_0(h)}e_3^{\psi_1(h)}e_5^{\psi_{01}(h)} \right) = \sigma(\psi_0,\psi_1)(g,h)\  \left( e_1^{\psi_0(gh)}e_3^{\psi_1(gh)}e_5^{\psi_{01}(gh)} \right).
\end{align}
Using the anticommutation relation of the $e_j$'s we obtain the following equations:
\begin{align}
    \left( e_1^{\psi_0(g)}e_3^{\psi_1(g)}e_5^{\psi_{01}(g)} \right) & \left( e_1^{\psi_0(h)}e_3^{\psi_1(h)}e_5^{\psi_{01}(h)} \right)  \\
 & =   (-1)^{\psi_0(h)\psi_{01}(g)+\psi_0(h)\psi_{1}(g)} \left( e_1^{\psi_0(g)} e_1^{\psi_0(h)}e_3^{\psi_1(g)}e_5^{\psi_{01}(g)} e_3^{\psi_1(h)}e_5^{\psi_{01}(h)} \right)\\
 & =   (-1)^{\psi_0(h)\psi_{0}(g)+\psi_1(h)\psi_{01}(g)} \left(e_1^{\psi_0(gh)} e_3^{\psi_1(g)}e_3^{\psi_1(h)}e_5^{\psi_{01}(g)} e_5^{\psi_{01}(h)} \right) \\
 & =   (-1)^{\psi_0(g)\psi_{0}(h)+\psi_0(g)\psi_{1}(h)+\psi_1(g)\psi_{1}(h)} \left(e_1^{\psi_0(gh)} e_3^{\psi_1(gh)}e_5^{\psi_{01}(gh)} \right),
\end{align}
thus showing that the 2-cocycle is indeed: 
\begin{align}
    \rho(\psi_0,\psi_1)(g,h) = (-1)^{\psi_0(g)\psi_{0}(h)+\psi_0(g)\psi_{1}(h)+\psi_1(g)\psi_{1}(h)}.
\end{align}

\end{proof}

Define the group: 
\begin{align}
    H^2(G, \mathbb{C}^*_\phi) \overset{\bullet}{\times}_\rho \mathrm{Hom}(G,\mathbb{Z}/2)
\end{align}
as the set $ H^2(G, \mathbb{C}^*_\phi) {\times} \mathrm{Hom}(G,\mathbb{Z}/2)$
with product structure:
\begin{align}
(\lambda_0, \psi_0)(\lambda_1, \psi_1) := ((-1)^{\psi_0 \cup \psi_1}\lambda_0\lambda_1 , \psi_0+\psi_1)
\end{align}
where $\psi_0 \cup \psi_1 \in H^2(G, \mathbb{Z}/2)$ is the cup product of $\psi_0$ with $\psi_1$ defined by the formula:
\begin{align}
   ( \psi_0 \cup \psi_1 ) (g,h) = \psi_0(g)\psi_1(h)
\end{align}
for $g,h \in G$. Then we see that the product: 
\begin{align}
    (\lambda_0, \psi_0)(\lambda_1, \psi_1)(\lambda_2, \psi_0+\psi_1) &=  ((-1)^{\psi_0 \cup \psi_1}\lambda_0\lambda_1 , \psi_0+\psi_1)(\lambda_2 , \psi_0+\psi_1) \\
    &=((-1)^{\psi_0\cup \psi_0 + \psi_0 \cup \psi_1 +\psi_1 \cup \psi_1}\lambda_0 \lambda_1\lambda_2, 0)
\end{align}
encodes the cohomology class of the cocycle $\rho(\psi_0,\psi_1) $ of Eqn. \eqref{formula for cocycle rho}, i.e.
\begin{align}
  [\rho(\psi_0,\psi_1)]=  (-1)^{\psi_0\cup \psi_0 + \psi_0 \cup \psi_1 +\psi_1 \cup \psi_1} \in H^2(G, \mathbb{C}^*_\phi).
\end{align}

\begin{corollary} \label{corollary GrBr'}
The assignment:
\begin{align}
   \Big( H^2(G, \mathbb{C}^*_\phi) \overset{\bullet}{\times}_\rho \mathrm{Hom}(G,\mathbb{Z}/2) \Big) &\overset{\cong}{\to} \mathrm{GrBr}_{(G,\phi)}(\mathbb{C})'\\
    (\lambda,\psi) & \mapsto \mathrm{End}(V_\lambda) \hat{\otimes} \sigma(\psi)
\end{align}
is an isomorphism of groups.
\end{corollary}
\begin{proof}
    The result follows from Props. \ref{proposition extensions} and \ref{proposition 2-cocycle rho},
    and the fact
    that the algebra $\mathrm{End}(V_{\rho(\psi_0,\psi_1)})$ is similar to the algebra $\rho(\psi_0) \hat{\otimes}\rho(\psi_1) \hat{\otimes}\rho(\psi_{01})^{-1} $
     since they define the same classes in $\mathrm{GrBr}_{(G,\phi)}(\mathbb{C})''$  and by the isomorphism of Eqn. \eqref{GrBr'' = H2} they induce the same classes in $H^2(G, \mathbb{C}^*_\phi)$.
\end{proof}

Now, the 2-cocycle: 
\begin{align}
    \beta : \mathbb{Z}/2 \times \mathbb{Z}/2  \to \mathrm{GrBr}_{(G,\phi)}(\mathbb{C})'
\end{align}
of the extension of Eqn. \eqref{short exact sequence GrBR(G)}
measures the failure of the section:
\begin{align}
 \alpha :   \mathbb{Z}/2 &\to \mathrm{GrBr}_{(G,\phi)}(\mathbb{C})\\
    0 & \mapsto (\mathbb{C}, g \mapsto \mathrm{Ad}_{\mathbb{K}^{\phi(g)}}) \\
    1 &\mapsto (C^{1,0}_\mathbb{R}\underset{\mathbb{R}}{\otimes} \mathbb{C},g \mapsto \mathrm{Ad}_{\mathbb{K}^{\phi(g)}})
\end{align}
to be a homomorphism. Here we are using
the complexification of the Clifford algebra appearing in Eqn. \eqref{complexification C10}.
The 2-cocycle $\beta$ is non-trivial only for $\beta(1,1)$, and in this case
we have that:
\begin{align}
    \beta(1,1) = \alpha(1) \hat{\otimes} \alpha(1) \cong (C^{2,0}_\mathbb{R}\underset{\mathbb{R}}{\otimes} \mathbb{C},g \mapsto \mathrm{Ad}_{e_1^{\phi(g)}\mathbb{K}^{\phi(g))}}).
\end{align}
Note that:
\begin{align}
    \beta(1,1) = \sigma(\phi)
\end{align}
where $\sigma(\phi)$ is defined in Eqns. 
\eqref{sigma of psi}, \eqref{Algebra for a homomorphism psi} and \eqref{Algebra for a homomorphism psi2}. Therefore we may take:
 \begin{align}
     \beta(a_0,a_1) = \sigma(\phi^{a_0a_1}).
 \end{align}

Define the group:
\begin{align}
    \mathrm{GrBr}_{(G,\phi)}(\mathbb{C})' \overset{\bullet}{\times}_\beta \mathbb{Z}/2
\end{align}
as the set $\mathrm{GrBr}_{(G,\phi)}(\mathbb{C})' {\times} \mathbb{Z}/2$
with group structure:
\begin{align}
 (A_0,a_0)(A_1,a_1) &:= (A_1 \hat{\otimes}A_2 \hat{\otimes} \beta(a_0,a_1), a_0+a_1) \\
 &= (A_1 \hat{\otimes}A_2 \hat{\otimes} \sigma(\phi^{a_0a_1}), a_0+a_1)
 \end{align}
 then we obtain the following result.

\begin{corollary} \label{corollary GrBr}
    The assignment: 
    \begin{align}
  \Big(       \mathrm{GrBr}_{(G,\phi)}(\mathbb{C})' \overset{\bullet}{\times}_\beta \mathbb{Z}/2 \Big) & \overset{\cong}{\to}  \mathrm{GrBr}_{(G,\phi)}(\mathbb{C}) \\
  (A,a) & \mapsto A \hat{\otimes} \alpha(a)
    \end{align}
    is an isomorphism of groups.
\end{corollary}

Note that
    $\sigma(\phi)$ in $\mathrm{GrBr}_{(G,\phi)}(\mathbb{C})'$ corresponds to $(1, \phi)$ in $H^2(G, \mathbb{C}^*_\phi) \overset{\bullet}{\times}_\rho \mathrm{Hom}(G,\mathbb{Z}/2) $. Therefore we may define the group:
 \begin{align}
    \Big( H^2(G, \mathbb{C}^*_\phi) \overset{\bullet}{\times}_\rho \mathrm{Hom}(G,\mathbb{Z}/2) \Big) \overset{\bullet}\times_\beta \mathbb{Z}/2 
\end{align}
as the set $H^2(G, \mathbb{C}^*_\phi) {\times} \mathrm{Hom}(G,\mathbb{Z}/2)  \times \mathbb{Z}/2$
with product structure:
\begin{align}
    (\lambda_0, \psi_0, a_0)(\lambda_1, \psi_1, a_1) := ((-1)^{\psi_0 \cup \psi_1 + \psi_0 \cup \phi^{a_0a_1}+\psi_1 \cup \phi^{a_0a_1}}\lambda_0\lambda_1,\psi_0+\psi_1+\phi^{a_0a_1}, a_0+a_1)
\end{align}

Now we can now bundle up corollaries \ref{corollary GrBr'} and \ref{corollary GrBr} to obtain the following result.

\begin{theorem} \label{Theorem decomposition GrBr}
    The assignment: 
\begin{align}
   \Big( H^2(G, \mathbb{C}^*_\phi) \overset{\bullet}{\times}_\rho \mathrm{Hom}(G,\mathbb{Z}/2) \Big) \overset{\bullet}\times_\beta \mathbb{Z}/2 &\overset{\cong}{\to} \mathrm{GrBr}_{(G,\phi)}(\mathbb{C})\\
    (\lambda,\psi,a)) & \mapsto \mathrm{End}(V_\lambda) \hat{\otimes} \sigma(\psi) \hat{\otimes} \alpha(a)
    \label{assigmnment theorem}
\end{align}
    is an isomorphism of groups.
\end{theorem}

\begin{proof}
    We just need to show that the assignment is a homomorphism of groups. 
We have on the one side that:
\begin{align}
    (\lambda_0, \psi_0, a_0)(\lambda_1, \psi_1, a_1) = ((-1)^{\psi_0 \cup \psi_1 + \psi_0 \cup \phi^{a_0a_1}+\psi_1 \cup \phi^{a_0a_1}}\lambda_0\lambda_1,\psi_0+\psi_1+\phi^{a_0a_1}, a_0+a_1)
\end{align}
     is mapped to: 
    \begin{align} \label{expression morphism+product}
        \mathrm{End}(V_{(-1)^{\psi_0 \cup \psi_1 + \psi_0 \cup \phi^{a_0a_1}+\psi_1 \cup \phi^{a_0a_1}}\lambda_0\lambda_1}) \hat{\otimes} \sigma(\psi_0+\psi_1 +\phi^{a_0a_1}) \hat{\otimes} \alpha(a_0+a_1);
    \end{align}
   while on the other we have the product:
   \begin{align} \label{expression producto+morphism}
       \mathrm{End}(V_{\lambda_0}) \hat{\otimes} \sigma(\psi_0) \hat{\otimes} \alpha(a_0) \hat{\otimes}
       \mathrm{End}(V_{\lambda_1}) \hat{\otimes} \sigma(\psi_1) \hat{\otimes} \alpha(a_1).
   \end{align}
    Whenever $a_0$ or $a_1$ equal $0$, then the equality of the expressions 
    \eqref{expression morphism+product} and \eqref{expression producto+morphism} follows from Cor. \ref{corollary GrBr'}.
    If both $a_0$ and $a_1$ equal $1$, then $a_0+a_1=0$, and: 
    \begin{align}
       \mathrm{End}(V_{(-1)^{ \psi_0 \cup \phi+\psi_1 \cup \phi}}) \hat{\otimes} \sigma(\psi_0+\psi_1+\phi)  &= \sigma(\psi_0 +\psi_1) \hat{\otimes} \sigma(\phi) \\ 
 &=       \sigma(\psi_0 +\psi_1) \hat{\otimes} \alpha(1) \hat{\otimes} \alpha(1)
    \end{align}
together with: 
\begin{align}
     \mathrm{End}(V_{(-1)^{ \psi_0 \cup \psi_1}}) \hat{\otimes} \sigma(\psi_0+\psi_1)  &= \sigma(\psi_0) \hat{\otimes} \sigma(\psi_1). 
\end{align}
Assembling these equalities we get the equality:
\begin{align}
    \mathrm{End}(V_{(-1)^{\psi_0 \cup \psi_1 + \psi_0 \cup \phi^{a_0a_1}+\psi_1 \cup \phi^{a_0a_1}}\lambda_0\lambda_1}) \hat{\otimes} \sigma(\psi_0+\psi_1 +\phi^{a_0a_1}) \hat{\otimes} \alpha(a_0+a_1) & \\
    = \mathrm{End}(V_{\lambda_0}) \hat{\otimes} \sigma(\psi_0) \hat{\otimes} \alpha(a_0) \hat{\otimes}
       \mathrm{End}(V_{\lambda_1}) \hat{\otimes} &\sigma(\psi_1)  \hat{\otimes} \alpha(a_1),
\end{align}
    which implies that the assignment of \eqref{assigmnment theorem} is a homomorphism. The injectivity and surjectivity follow from corollaries \ref{corollary GrBr'} and \ref{corollary GrBr}.
\end{proof}

\subsection{Relation to other equivariant graded Brauer groups}

The explicit decomposition of the magnetic equivariant graded Brauer group presented in Thm. \ref{Theorem decomposition GrBr} fits within a series of decomposition results of Brauer groups of different kinds. Here we will survey some of these decompositions.

Using Galois cohomology it can be shown that
the Bauer group of a field $k$ can be described in cohomological terms as follows
\cite[Thm. 4.4.7]{Gille_Central_Simple_algebras}:
\begin{align}
    \mathrm{Br}(k) \cong H^2(\mathrm{Gal}(k_s/k), k_s^*)
\end{align}
where $k_s$ is the separable closure of $k$ and the absolute Galois group $\mathrm{Gal}(k_s/k)$ acts on the coefficients $k_s^*$.

The equivariant extension of the Brauer group is due to Fr\"ohlich, and in \cite[Thm. 4.1]{Froelich} it is shown the following isomorphism:
\begin{align}
    \mathrm{Br}_G(k) \cong \mathrm{Br}(k) \times H^2(G,k^*)
\end{align}
where $G$ acts trivially on $k^*$.

The extension to graded algebras was due to Wall and the graded Brauer group was born (also known as Brauer-Wall group). In \cite[Thm. 3]{Wall_Graded_Brauer_groups} the extension of groups is shown, cf. \cite[Thm. 4.4]{Lam_Introduction_to_quadratic_forms}:
\begin{align}
 0 \to  \mathrm{Br}(k) \to   \mathrm{GrBr}(k) \to  \left((k^*/k^{*2}) \overset{\cdot}{\times} \mathbb{Z}/2\right) \to 0
\end{align}
where the product structure on $(k^*/k^{*2}) \overset{\cdot}{\times} \mathbb{Z}/2$ is given by the formula: 
\begin{align}
    (d,e)(d',e')=((-1)^{ee'}dd',e+e'),
\end{align}
and the 2-cocycle of the extension has a similar description than the one in Prop. \ref{proposition 2-cocycle rho}.

The equivariant graded Brauer group was introduced by Riehm, and in \cite[Thm. 4]{Riehm_Graded_Equivariant_CSA} the following isomorphism is shown:
\begin{align} \label{equivariant graded brauer group}
    \mathrm{GrBr}_G(k) \cong \Big( H^2(G,k^*) \overset{\bullet}{\times}  \mathrm{Hom}(G, \mathbb{Z}/2) \Big) \times \mathrm{GrBr}(k) 
\end{align}
where the twisted product $\overset{\bullet}{\times}$ is defined by the 2-cocycle which is the composition of the cup product with the homomorphism induced by the inclusion of coefficients $\mathbb{Z}/2 \to k^*$:
\begin{align}
    \mathrm{Hom}(G, \mathbb{Z}/2) \times \mathrm{Hom}(G, \mathbb{Z}/2) &\overset{\cup}{\to} H^2(G,\mathbb{Z}/2)\to H^2(G,k^*) \nonumber\\
    (\psi_0,\psi_1)& \mapsto \psi_0 \cup \psi_1 \mapsto (-1)^{\psi_0\cup \psi_1}.
\end{align}

Applying the decomposition formula of Eqn. \eqref{equivariant graded brauer group} to the cases of $k= \mathbb{R}, \mathbb{C}$ we obtain the following formulas:
\begin{align}
    \mathrm{GrBr}_G(\mathbb{R}) & \cong \Big( H^2(G,\mathbb{Z}/2) \overset{\bullet}{\times}  \mathrm{Hom}(G, \mathbb{Z}/2) \Big) \times \mathbb{Z}/8, \\
    \mathrm{GrBr}_G(\mathbb{C}) & \cong \Big( H^2(G,\mathbb{C}^*) \overset{\bullet}{\times}  \mathrm{Hom}(G, \mathbb{Z}/2) \Big) \times \mathbb{Z}/2.
\end{align}

Relating these two decompositions of the graded equivariant Brauer groups to the magnetic equivariant graded Brauer groups we get the following homomorphisms.

Take the magnetic group $(G_0 \times \mathbb{Z}/2, \pi_2)$ and note that the complexification of real algebras induces the homomorphism:
\begin{align}
    \mathrm{GrBr}_{G_0}(\mathbb{R}) \to  \mathrm{GrBr}_{(G_0\times \mathbb{Z}/2, \pi_2)}(\mathbb{C}).
\end{align}
This homomorphism fails to be an isomorphism for many groups, and the explicit examples will be shown in the next section.

If $G_0 = \mathrm{ker}(G \overset{\phi}{\to} \mathbb{Z}/2)$ is the core of the magnetic group $G$, the restriction of the $G$ action to $G_0$ induces a homomorphism of groups:
\begin{align}
    \mathrm{GrBr}_{(G,\phi)}(\mathbb{C}) \to  \mathrm{GrBr}_{G_0}(\mathbb{C})
\end{align}
which is surjective whenever the homomorphism:
\begin{align}
    H^2(G,\mathbb{C}^*_{\phi}) \to H^2(G_0,\mathbb{C}^*)
\end{align}
is surjective.  The explicit examples will be shown in the next section.





\subsection{Examples}

Prior to introducing the examples, let us start by recalling some well known facts on the cohomology of groups.

Consider the short exact sequence of coefficients:
\begin{align}
    0 \to \mathbb{Z}_\phi \overset{\times 2\pi i}\longrightarrow \mathbb{C}_\phi \overset{\exp}\longrightarrow \mathbb{C}_\phi^* \to 0
\end{align}
compatible with the $G$-action induced by the homomorphism $\phi$, that is, on $\mathbb{Z}_\phi$ and $\mathbb{C}_\phi$ it acts by the sign representation, and on $\mathbb{C}_\phi^*$ by complex conjugation.

We know that $H^j(G, \mathbb{C}_\phi)=0$ for $j \geq 1$ and therefore the long exact sequence in coefficients imply that:
\begin{align}
    H^j(G, \mathbb{C}_\phi^*) \cong H^{j+1}(G, \mathbb{Z}_\phi) \ \ \ \mathrm{for} \ \ \ j \geq 1.
\end{align}

Take now $C$ to be a finite cyclic group and $M$ any abelian group with an action of $C$;    this is the same as a $\mathbb{Z}C$-module.

From \cite[p.59]{Brown_cohomology_of_groups} we know that the cohomology of the cyclic group $C$ with coefficients in the module $M$
can be determined for $k \geq 1$ by the isomorphisms:
\begin{align}
    H^{0}(C,M) & \cong M^C\\
    H^{2k}(C,M) & \cong \mathrm{coker}(M_C \to M^C)  \ \ \ \mathrm{for}   \ \ k \geq 1\\
    H^{2k-1}(C,M) &\cong \mathrm{ker}(M_C \to M^C)   \ \ \ \mathrm{for}   \ \ k \geq 1
\end{align}
where the invariants  and the coinvariants  are respectively defined as follows:
\begin{align}
    M^C &= \{ m \in M | \forall g \in C \ gm=m \} \\
    M_C &= M / \langle gm-m | g \in C, m \in M \rangle
\end{align}
and the homomorphism: 
\begin{align}
    M_C \to M^C, \ \ [m]\mapsto \sum_{g \in C} gm
\end{align} is called the norm map.

For the group $\mathbb{Z}/2$ we have the coefficients $\mathbb{Z}$ with trivial action and  
$\mathbb{Z}_-$ with the sign representation.
Here we have $\mathbb{Z}_{\mathbb{Z}/2}=\mathbb{Z}$, 
$\mathbb{Z}^{\mathbb{Z}/2}=\mathbb{Z}$,
$(\mathbb{Z}_-)_{\mathbb{Z}/2}=\mathbb{Z}/2$, 
$(\mathbb{Z}_-)^{\mathbb{Z}/2}=0$ and therefore for $k \geq 1$ we have the isomorphism:
\begin{align}
    H^{2k}(\mathbb{Z}/2, \mathbb{Z})  \cong H^{2k-1}(\mathbb{Z}/2, \mathbb{Z}_-) &\cong \mathbb{Z}/2,\\
    H^{2k-1}(\mathbb{Z}/2, \mathbb{Z})  = H^{2k}(\mathbb{Z}/2, \mathbb{Z}_-) &=0.
\end{align}

To calculate $H^2(G, \mathbb{C}_\phi^*)$ for the magnetic group $(G, \phi)$ we use the Lyndon-Hochschild-Serre (LHS) spectral sequence associated to the short exact sequence of groups:
\begin{align}
G_0 \to G \overset{\phi}{\to} \mathbb{Z}/2.
\end{align}
The second page of the spectral sequence gives us the terms:
\begin{align}
    E_2^{2,0} = &H^2(\mathbb{Z}/2, \mathbb{C}^*_{\mathrm{id}}) \cong \mathbb{Z}/2, \\ 
    E_2^{1,1} = & H^1(\mathbb{Z}/2, H^1(G_0,\mathbb{C}^*_\phi)), \\
    E_2^{0,2} = & {H^2(G_0, \mathbb{C}^*_\phi)}^{\mathbb{Z}/2},
\end{align}
and the second differential $d_2:E_2^{p,q} \to E_2^{p+2,q-1}$ is the only differential that might be non-zero. This follows from the fact that the term: 
\begin{align}    
E_2^{3,0} \cong H^3(\mathbb{Z}/2, \mathbb{C}^*_{\mathrm{id}}) \cong H^4(\mathbb{Z}/2, \mathbb{Z}_-) = 0
\end{align}
and therefore $E_3^{p,q}=E_\infty^{p,q}$ for $p+q=2$.

\begin{example} [Cyclic magnetic groups]
    Consider the magnetic cyclic group $(G,\phi)=(\mathbb{Z}/(2n),\mathrm{mod}_2)$.
 We have in this case $(\mathbb{Z}_{-})^{\mathbb{Z}/2n}=0$  and 
 $(\mathbb{Z}_{-})_{\mathbb{Z}/2n} = \mathbb{Z}/2$,
 and therefore:
\begin{align}
    H^2(\mathbb{Z}/2n,\mathbb{C}^*_{\phi}) = & H^3(\mathbb{Z}/2n,\mathbb{Z}_{-}) \\ = & \mathrm{ker}\left( (\mathbb{Z}_{-})_{\mathbb{Z}/2n}\to (\mathbb{Z}_{-})^{\mathbb{Z}/2n}\right) \\ 
    = & \mathbb{Z}/2.
\end{align}

Note that $H^2(\mathbb{Z}/2n,\mathbb{Z}/2)= \mathbb{Z}/2$ and the inclusion of coefficients $\mathbb{Z}/2 \to \mathbb{C}^*_\phi$ induces an isomorphism:
\begin{align}
    H^2(\mathbb{Z}/2n,\mathbb{Z}/2) \overset{\cong}{\to}
    H^2(\mathbb{Z}/2n,\mathbb{C}^*_\phi).
\end{align}

The non-trivial element in $H^2(\mathbb{Z}/2n,\mathbb{Z}/2) $ defines the $\mathbb{Z}/2$-central extension of $\mathbb{Z}/2n$ given by:
\begin{align}
    \mathbb{Z}/2 \to \mathbb{Z}/4n \to \mathbb{Z}/2n.
\end{align}

If we consider any surjective homomorphism $\mu:\mathbb{Z}/2n \to \mathbb{Z}/2k$ 
of the magnetic groups $(\mathbb{Z}/2n, \mathrm{mod}_2)$ and $(\mathbb{Z}/2k, \mathrm{mod}_2)$ with $k|n$,
we have that the pullback of $\mathbb{Z}/4k$ by $\mu$ depends of the parity of $\tfrac{n}{k}$. Explicitly we have:
\begin{align}
    \mu^*(\mathbb{Z}/4k)
    \cong 
    \begin{cases}
    \mathbb{Z}/2 \times \mathbb{Z}/2n & \tfrac{n}{k} \ \mathrm{even},\\
     \mathbb{Z}/4n & \tfrac{n}{k} \ \mathrm{odd}.
    \end{cases}
\end{align}
and therefore the pullback homomorphism:
\begin{align}
    \mu^*: H^2(\mathbb{Z}/2k, \mathbb{Z}/2) \to H^2(\mathbb{Z}/2n, \mathbb{Z}/2)
\end{align}
is trivial if $\tfrac{n}{k}$ is even, and an isomorphism if $\tfrac{n}{k}$ is odd. Since the class in $H^2(\mathbb{Z}/2, \mathbb{Z}/2)$ is generated by the square of the non-trivial  class in $\mathrm{Hom}(\mathbb{Z}/2, \mathbb{Z}/2)$, we have that the square of the 
non trivial class in $\mathrm{Hom}(\mathbb{Z}/2n, \mathbb{Z}/2)$ trivializes in $H^2(\mathbb{Z}/2n, \mathbb{Z}/2)$ whenever $n$ is even, and it is the generator of the group if $n$ is odd. We conclude that:
\begin{align} \label{GrBr(Z2n)}
     \mathrm{GrBr}_{(\mathbb{Z}/2n,\mathrm{mod}_2)}(\mathbb{C}) =
     \begin{cases}
         \mathbb{Z}/8 &   n \ \mathrm{odd} \\
         \mathbb{Z}/2 \times \mathbb{Z}/4 &   n \ \mathrm{even},
     \end{cases}
\end{align}
and the induced homomorphisms for $1 \leq k$ are the following:
\begin{align}
    \mathrm{GrBr}_{(\mathbb{Z}/2,\mathrm{id})}(\mathbb{C}) &\to \mathrm{GrBr}_{(\mathbb{Z}/4k,\mathrm{mod}_2)}(\mathbb{C})& \mathrm{GrBr}_{(\mathbb{Z}/2,\mathrm{mod}_2)}(\mathbb{C}) & \stackrel{\cong}{\to} \mathrm{GrBr}_{(\mathbb{Z}/2(2k+1),\mathrm{mod}_2)}(\mathbb{C}) \\
    \mathbb{Z}/8 &\to \mathbb{Z}/2 \times \mathbb{Z}/4& \mathbb{Z}/8 &\stackrel{=}{\to} 
     \mathbb{Z}/8 \\
     a & \mapsto (0,a) & & 
\end{align}

Note that the generator of the Brauer group $\mathrm{Br}_{(\mathbb{Z}/2,\mathrm{id})}(\mathbb{C})$
can be taken to be the algebra:
\begin{align}
\left(M_2(\mathbb{C}),  \mathrm{Ad}_{\left(\begin{smallmatrix}
            0 & -1 \\ 1 & 0
        \end{smallmatrix}\right) \mathbb{K}}\right),
\end{align}
being the image of $\mathbb{H}$ in the Brauer group $\mathrm{Br}(\mathbb{R})$ after complexification, as it is shown in Eqn. \eqref{M2(C) generator Br(R)}. The magnetic projective
representation of $\mathbb{Z}/2$ that the operator ${\left(\begin{smallmatrix}
            0 & -1 \\ 1 & 0
        \end{smallmatrix}\right)} \mathbb{K}$ defines,
        is not projective once it is being pulled back to $\mathbb{Z}/4k$. This shows that the morphism:
\begin{align}
   \mathbb{Z}/2 \cong  \mathrm{Br}_{(\mathbb{Z}/2,\mathrm{id})}(\mathbb{C}) \overset{\times0}{\to} \mathrm{Br}_{(\mathbb{Z}/4k,\mathrm{mod}_2)}(\mathbb{C}) \cong \mathbb{Z}/2
\end{align}
        is trivial.
\end{example}

\begin{example}[Direct and semidirect product by $\mathbb{Z}/2$]    
Consider the magnetic group of the form $(G,\phi)=(G_0\rtimes \mathbb{Z}/2,\pi_2)$. 
In this case the homomorphisms of magnetic groups:
\begin{align}
    (\mathbb{Z}/2, \mathrm{id}) \overset{\iota}{\to} (G_0 \rtimes \mathbb{Z}/2, \pi_2) \overset{\pi_2}{\to} (\mathbb{Z}/2, \mathrm{id}) 
\end{align}
with $\iota(a)=(1_G,a)$ and $\pi_2 \circ \iota = \mathrm{id}_{\mathbb{Z}/2}$, implies that:
\begin{align}
   \mathbb{Z}/8 \cong \mathrm{GrBr}_{(\mathbb{Z}/2,\mathrm{id})}(\mathbb{C}) \subset \mathrm{GrBr}_{(G_0 \rtimes\mathbb{Z}/2,\pi_2)}(\mathbb{C}). 
\end{align}
The second differential of the LHS spectral sequence is trivial for the elements with $p+q \leq 2$ and we have that:
\begin{align}\label{extension probem for products}
0 \to 
H^1(\mathbb{Z}/2, H^1(G_0, \mathbb{C}^*_\phi)) \to
    \tfrac{H^2(G_0 \rtimes \mathbb{Z}/2, \mathbb{C}^*_\phi)}{H^2(\mathbb{Z}/2, \mathbb{C}^*_{\mathrm{id}})}  \to H^2(G_0, \mathbb{C}^*_\phi)^{\mathbb{Z}/2} \to 0
\end{align}

Let us take $G_0=\mathbb{Z}/2n$ and the direct product $\mathbb{Z}/2n \times \mathbb{Z}/2$ with $\phi$ the projection on the second coordinate. In this case we have:
\begin{align}
    H^2(\mathbb{Z}/2n \times \mathbb{Z}/2, \mathbb{C}^*_\phi) \cong H^2( \mathbb{Z}/2, \mathbb{C}^*_{\mathrm{id}}) \times H^1(\mathbb{Z}/2, \widehat{\mathbb{Z}/2n}_-)
\end{align}
where $\widehat{\mathbb{Z}/2n} \cong H^1(\mathbb{Z}/2n, \mathbb{C}^*)$ and the $\mathbb{Z}/2$-action conjugates the homomorphisms in $\widehat{\mathbb{Z}/2n}$.

The non-trivial twisted homomorphism in:
\begin{align}
    H^1(\mathbb{Z}/2, \widehat{\mathbb{Z}/2n}_-) \cong \mathbb{Z}/2
\end{align}
corresponds to the projective representation of the magnetic group generated by the matrices:
\begin{align}
  (1,0) \mapsto  \left(\begin{smallmatrix}
           1& 0  \\  0 & -1
        \end{smallmatrix}\right), \ \ \  \ \ (0,1) \mapsto
        \left(\begin{smallmatrix}
            0 & 1 \\ 1 & 0
        \end{smallmatrix}\right) \mathbb{K}.
\end{align}
 Note that the commutator of these two matrices is a non-trivial element
in the center of $M_2(\mathbb{C})$; hence it is  a projective representation of the magnetic group $(\mathbb{Z}/2n \times \mathbb{Z}/2,\pi_2)$.

Since $H^2(\mathbb{Z}/2n,\mathbb{Z}/2)=\mathbb{Z}/2$, we obtain the following homomorphisms:
\begin{align}
\mathrm{GrBr}_{\mathbb{Z}/2n}(\mathbb{R})  &\to 
\mathrm{GrBr}_{(\mathbb{Z}/2n\times \mathbb{Z}/2, \pi_2)}(\mathbb{C})\\
{\scriptstyle H^2(\mathbb{Z}/2n,\mathbb{Z}/2) \times \mathrm{Hom}(\mathbb{Z}/2n\,\mathbb{Z}/2) \times \mathbb{Z}/8} & \to
{\scriptstyle H^1(\mathbb{Z}/2, \widehat{\mathbb{Z}/2n}_-) \times \mathrm{Hom}(\mathbb{Z}/2n\,\mathbb{Z}/2) \times \mathbb{Z}/8} \\ 
\mathbb{Z}/2 \times \mathbb{Z}/2 \times \mathbb{Z}/8 & \to \mathbb{Z}/2 \times \mathbb{Z}/2 \times \mathbb{Z}/8 \\
(a,b,c) & \mapsto (0,b,c).
\end{align}

Forgetting further the antilinear automorphisms we get the homomorphism:
\begin{align} 
\mathrm{GrBr}_{(\mathbb{Z}/2n\times \mathbb{Z}/2, \pi_2)}(\mathbb{C})
& \to \mathrm{GrBr}_{\mathbb{Z}/2n}(\mathbb{C}) \\
{\scriptstyle H^1(\mathbb{Z}/2, \widehat{\mathbb{Z}/2n}_-) \times \mathrm{Hom}(\mathbb{Z}/2n\,\mathbb{Z}/2) \times \mathbb{Z}/8} & \to {\scriptstyle H^2(\mathbb{Z}/2n,\mathbb{C}^*) \times \mathrm{Hom}(\mathbb{Z}/2n\,\mathbb{Z}/2) \times \mathbb{Z}/2} \\ 
\mathbb{Z}/2 \times \mathbb{Z}/2 \times \mathbb{Z}/8 & \to \mathbb{Z}/2 \times \mathbb{Z}/2 \times \mathbb{Z}/2 \\
(a,b,c) & \mapsto (0,b,c).
\end{align}

\begin{remark} The magnetic equivariant graded  Brauer group $\mathrm{GrBr}_{(\mathbb{Z}/2n\times \mathbb{Z}/2, \pi_2)}(\mathbb{C})$ is isomorphic to  the graded real equivariant Brauer group
$\mathrm{GrBr}_{\mathbb{Z}/2n}(\mathbb{R})$, but the complexification from the latter to the former is not an isomorphism. The real projective representation of $\mathbb{Z}/2n$ is not projective once complexified, and there is a magnetic projective representation of the magnetic group $\mathbb{Z}/2n \times \mathbb{Z}/2$ which is of complex type and does not appear as a complexification of a real projective representation.
\end{remark}
\end{example}

\subsection{Relation to Dyson's ten-fold classification}

In 1962, while studying symmetries of Hamiltonians, physicist F. Dyson argued that the three-fold  classification of associative division algebras by Frobenius should be enhanced by a ten-fold classification which incorporates the symmetries of graded vector spaces \cite{Dyson_Threefold}.

This ten-fold classification of Dyson is obtained when studying the algebra of endomorphisms  of irreducible graded representations of graded groups \cite{Moore-Dysons}. Let us recall the main constructions.

Let $(G,\phi)$ be a magnetic group, and endow it with an extra grading defined by a homomorphism $\psi: G \to \mathbb{Z}/2$. Let $V$ be a complex graded magnetic representation of $G$, meaning that $V= V_0 \oplus V_1$ is a graded vector space,
the elements of $\psi^{-1}(0)$ act by homogeneous endomorphisms of degree zero, while the elements of  
$\psi^{-1}(1)$ act by homogeneous endomorphisms of degree one, and the elements of $G_0=\phi^{-1}(0)$ act complex linearly while the elements of $G \backslash G_0$ ac complex antilinearly. 

The endomorphism algebra $\mathrm{End}(V)$ can be endowed with an action of $G$ by conjugation:
\begin{align}
    G \times \mathrm{End}(V) \to \mathrm{End}(V)\\
    (g, \Phi) \mapsto g \Phi g^{-1},
\end{align}
making it a magnetic equivariant central simple graded algebra, and thus defining  a similarity class
in $\mathrm{GrBr}_{(G,\phi)}(\mathbb{C})$.

If the representation $V$ is irreducible, the algebra of $G$-invariant endomorphisms 
$\mathrm{End}(V)^G$, which is the same as the $G$-equivariant endomorphisms, becomes a graded division algebra over the reals. The explicit classification of the ten possibilities was masterfully done by Dyson \cite{Dyson_Threefold} and was later related to the classification of the graded division algebras over the reals \cite{Moore-Dysons}.

Note two important facts. First, the algebra $\mathrm{End}(V)$ for $V$ irreducible of degree zero, is trivial in the graded Brauer group $\mathrm{GrBr}_{(G,\phi)}(\mathbb{C})$, independently of the magnetic group $G$. Second, for the case $G=\mathbb{Z}/2$ we have always that
$\mathrm{End}(V)^{\mathbb{Z}/2} \cong \mathbb{R}$.

Therefore we see that there is no way to extract information of the $G$-invariant algebras out of the magnetic equivariant graded Brauer groups, except perhaps for the magnetic group $\mathbb{Z}/2$. Since the algebra $\mathrm{End}(V)$
is trivial in $\mathrm{GrBr}_{(G,\phi)}(\mathbb{C})$, then one would expect that its $G$-invariant part would be trivial in $\mathrm{GrBr}(\mathbb{R})$. But as we have shown in Eqn. \eqref{wigenr irreps}, this is not possible.

The case of the magnetic group $\mathbb{Z}/2$ is special in the sense that taking $\mathbb{Z}/2$-invariants is the inverse map of the complexification isomorphism:
\begin{align}
    \mathrm{GrBr}(\mathbb{R}) \underset{\cong}{\overset{\otimes \mathbb{C}}{\longrightarrow}} \mathrm{GrBr}_{(\mathbb{Z}/2, \mathrm{id})}(\mathbb{C}) \underset{\cong}{\overset{\wedge(\mathbb{Z}/2)}{\longrightarrow}}  \mathrm{GrBr}(\mathbb{R}).
    \end{align}
For the magnetic group $\mathbb{Z}/2$ we can relate its graded equivariant Brauer groups with the set of real graded division algebras, as it is shown in Table \ref{Table-tenfold-way}. Meanwhile for a general magnetic group $(G,\phi)$, the magnetic equivariant graded Brauer group $\mathrm{GrBr}_{(G,\phi)}(\mathbb{C})$ does not relate in the same manner as for the group $\mathbb{Z}/2$ to the set of real graded division algebras.



\section{Twistings in magnetic equivariant K-theory} The graded Brauer group appears
as the natural group parametrizing twistings in K-theory. This has been masterfully shown by Donovan and Karoubi in \cite{Donovan_Karoubi_Graded_Brauer_k_theory}
where they study and analyze twisted K-theory as the appropriate K-theory with local coefficients, building on the theory of Banach categories developed by Karoubi in his thesis \cite{Karoubi-thesis}. 

The foundations, properties and results of Karoubi's approach to twisted K-theory provide an appropriate background where graded Brauer groups are the fundamental objects, and the  K-theories associated to them recover the well known K-theories built from vector bundles on spaces.
Here we will follow Karoubi's approach to distinguish the twists on magnetic equivariant K-theory \cite{serrano2025magneticequivariantktheory}. The main construction involves the theory of Banach categories which we will follow from Karoubi's book  \cite{karoubi}, and from the survey on the relation of twisted K-theory and the graded Brauer groups done by the same author  \cite{Karoubi_twisted_old_new}.

 Summarizing some of the results by Donovan and Karoubi \cite{Donovan_Karoubi_Graded_Brauer_k_theory} and incorporating some of the results of Atiyah and Segal \cite{AtiyahSegal-twisted-k-theory}, we have that the complex and real equivariant graded Brauer groups of a finite $G$-CW-complex $X$ become:
 \begin{align} \label{GrBrU}
     \mathrm{GrBrU}_{G}(X) &\cong \Big(\mathrm{Tors}(H^3_G(X,\mathbb{Z})) \overset{\bullet}{\times} H^1_G(X, \mathbb{Z}/2) \Big){\times} \mathbb{Z}/2, \\ \label{GrBrO}
     \mathrm{GrBrO}_{G}(X) &\cong\Big( H^2_G(X,\mathbb{Z}/2) \overset{\bullet}{\times} H^1_G(X, \mathbb{Z}/2) \Big) {\times} \mathbb{Z}/8.
 \end{align}
Here the $G$-equivariant cohomology groups $H^*_G$ stand for the cohomology of the homotopy quotient $X \times_GEG$.
These two groups parametrize respectively the twistings in the K-theory $\mathrm{KU}_G(X)$ of complex $G$-equivariant bundles over $X$ and in 
the K-theory $\mathrm{KO}_G(X)$ of real $G$-equivariant bundles over $X$.
In both cases the product structure $\overset{\bullet}{\times}$ incorporates the cup product in cohomology as in Cor. \ref{corollary GrBr'},
where in the unitary case it is further composed with the integer Bockstein map.

The authors together with M. Xicont\'encatl \cite{serrano2025magneticequivariantktheory} have introduced the magnetic equivariant K-theory $\mathbf{K}_{(G,\phi)}(X)$ as the K-theory group of complex vector bundles over $X$
endowed with actions of the magnetic group $(G,\phi)$.
The twistings of this K-theory have been assumed by other authors \cite{FreedMoore_Twisted_equivariant_matter, Gomi2017FreedMooreK} to behave in similar way as the ones of the K-theories $\mathrm{KU}$ and $\mathrm{KO}$ presented in Eqns. \eqref{GrBrU} and \eqref{GrBrO}. Here we will argue that
the appropriate graded equivariant Brauer groups for magnetic equivariant K-theory are  the magnetic equivariant graded Brauer groups.
In particular we will show that the magnetic equivariant graded Brauer group $\mathrm{GrBr}_{(G,\phi)}(\mathbb{C})$ parametrizes the twistings of the magnetic equivariant K-theory $\mathbf{K}_{(G,\phi)}(*)$ of a point.

One main application of this perspective is the alternative proof of the degree shift isomorphism \cite[Thm. 3.12]{Gomi2017FreedMooreK} presented in Prop. \ref{degree shift}, and its consequence, which is a 4-periodicity on the total magnetic equivariant K-theory for certain magnetic groups presented in Cor.
\ref{corollary 4-periodicity}.

\subsection{K-theory \`a la Karoubi}
Let $(A, \tau : G \to \mathrm{MAut}_{\mathrm{gr}}(A))$ be a magnetic $(G,\phi)$-equivariant central simple graded algebra, thus defining a similarity class in $\mathrm{GrBr}_{(G, \phi)}(\mathbb{C})$.  Following Karoubi \cite[\S 5]{Donovan_Karoubi_Graded_Brauer_k_theory}, consider the category:
\begin{align}
\mathcal{E}^{A}_{(G,\phi)}
\end{align}
whose objects are  $(G,\phi)$-equivariant $A$-modules
 and whose morphisms are degree zero morphisms of $A$-modules which are $G$-equivariant.  These modules are
 graded $A$-modules $V=V_0 \oplus V_1$ endowed
 with a compatible action of $G$ such that:
 \begin{align} \label{equivariant formula for action}
     g \cdot (av) = (g\cdot a)(g \cdot v) 
 \end{align}
 for every $g \in G$, $a \in A$ and $v \in V$
\cite[Def. 1.1]{Karoubi-Weibel}.
Let $\overline{\mathcal{E}}^{A}_{(G,\phi)}$ be the category whose objects are the same as in
${\mathcal{E}}^{A}_{(G,\phi)}$
but whose morphisms are not necessarily of degree zero.

The automorphisms of $A$ parametrized by $G$ can be lifted to the algebra:
\begin{align}
A^{p,q}:=A \hat{\otimes}C^{p,q}_\mathbb{C}
\end{align}
making it a magnetic $(G,\phi)$-equivariant central simple graded algebra. Here the $G$ action on $A \hat{\otimes}C^{p,q}_\mathbb{C}$ is: 
\begin{align}
g \cdot(a \hat{\otimes} b):= (g \cdot a) \hat{\otimes} (\mathbb{K}^{\phi(g)}b)
\end{align}
 for $a \in A$ and $b\in C^{p,q}_\mathbb{C}$.
 The forgetful functor: 
\begin{align} \label{forgetful_phi_p_q}
    \varphi^{p,q} : \mathcal{E}^{A^{p,q}}_{(G,\phi)} \to \overline{\mathcal{E}}^{A^{p,q}}_{(G,\phi)}
\end{align}
is a Banach functor between Banach categories \cite{Karoubi_Categories-Banach}, and the Grothendieck group:
\begin{align} \label{K-theory banach functor}
\mathbf{K}^{A^{p,q}}_{(G,\phi)}
\end{align}
is the K-group $K(\varphi^{p,q})$ of the Banach functor $\varphi^{p,q}$ \cite[pp. 191]{Karoubi-thesis}
\footnote{To be precise, this is the definition appearing in \cite[Rem. 2, pp 203]{Karoubi-thesis} and also appearing in \cite[\S 5]{Donovan_Karoubi_Graded_Brauer_k_theory}.}.
 An element in this K-theory group consists of the homotopy class of a triple $(E,F, \sigma)$
with $E,F$ objects in ${\mathcal{E}}^{A^{p,q}}_{(G,\phi)}$
and $\sigma : \varphi^{p,q} E \overset{\cong}{\to} \varphi^{p,q} F$ an isomorphism
not necessarily of degree zero. 
A triple $(E,F, \sigma)$ is called {\it elementary} if $\alpha$ is of degree zero. Two triples
$(E_0,F_0, \sigma_0)$ and $(E_1,F_1, \sigma_1)$
are {\it homotopic} if there exists two isomorphisms of degree zero $f:E_0 \to E_1$ and $g:F_0 \to F_1$, and a continuous map $\alpha: [0,1] \to \mathrm{Iso}_{{\overline{\mathcal{E}}}^{A^{p,q}}_{(G,\phi)}}(\varphi^{p,q}E_0,\varphi^{p,q}F_0)$ of isomorphisms such that $\alpha(0)=\sigma_0$, $\alpha(1)=g^{-1} \sigma_1 f$ \cite[\S 1.4]{Karoubi_Categories-Banach}.
We set: 
\begin{align}
\mathbf{K}^{A}_{(G,\phi)}:= \mathbf{K}^{A^{0,0}}_{(G,\phi)}.
\end{align}

If $V=V_0 \oplus V_1$ is a graded magnetic representation of $(G, \phi)$, the functor:
\begin{align}
    \gamma : {\mathcal{E}}^{A}_{(G,\phi)}  \to {\mathcal{E}}^{A \hat{\otimes}\mathrm{End}(V)}_{(G,\phi)} \ \ \ F \mapsto F \hat{\otimes}V
\end{align}
is an equivalence of Banach categories \cite[Thm. 16]{Donovan_Karoubi_Graded_Brauer_k_theory}, and therefore it induces an isomorphism in K-theories
\begin{align} \label{iso ktheories similar algebras}
    \gamma_* : \mathbf{K}^{A}_{(G,\phi)} \overset{\cong}{\to} \mathbf{K}^{A \hat{\otimes}\mathrm{End}(V)}_{(G,\phi)}.
\end{align}

The previous isomorphism of Eqn. \eqref{iso ktheories similar algebras} shows that the K-theory thus defined is independent of the algebra in the similarity class (see Eqn. \eqref{eqn tensor reps}). Following  \cite[Def. 17] {Donovan_Karoubi_Graded_Brauer_k_theory} we can define:
\begin{definition}
    For $\alpha \in \mathrm{GrBr}_{(G,\phi)}(\mathbb{C})$
we define $\mathbf{K}_{(G,\phi)}^\alpha$ (up to canonical isomorphism) as the Grothendieck group $\mathbf{K}_{(G,\phi)}^A$ for any $A$ in the class of $\alpha.$ The complete K-theory groups are then:
\begin{align}
    \mathbf{K}_{(G,\phi)}^{\mathrm{tot}} := \bigoplus_{\alpha \in \mathrm{GrBr}_{(G,\phi)}(\mathbb{C})} \mathbf{K}_{(G,\phi)}^\alpha 
\end{align}
\end{definition}

\subsection{Twisted magnetic equivariant K-theory}
In order to recover the magnetic equivariant K-theory groups of a point defined in \cite{serrano2025magneticequivariantktheory}, we will first introduce some alternative forms of understanding
$(G,\phi)$-equivariant $A$-modules. Take the twisted algebra $A \rtimes G$ as the algebra of linear combinations of elements $(a, g)$ with $a \in A$ and $g \in G$ with product: 
\begin{align}
    (a,g)(b, h)=(a (g\cdot b),gh),
\end{align}
and consider the category $\mathcal{E}^{A \rtimes G}$ of graded modules $M$ of this
algebra where: 
\begin{align}
    (a,g)\lambda m = (\mathbb{K}^{\phi(g)}\lambda) (a,g) m
\end{align}
for $\lambda \in \mathbb{C}$ and $m \in M$.

\begin{lemma}
    The categories
${\mathcal{E}}^{A}_{(G,\phi)} $ and ${\mathcal{E}}^{A \rtimes G}$ are naturally equivalent. Hence: 
\begin{align}
\mathbf{K}^{A}_{(G,\phi)}\cong \mathbf{K}^{A \rtimes G}.
\end{align}
\end{lemma}
\begin{proof}
The categories
${\mathcal{E}}^{A}_{(G,\phi)} $ and ${\mathcal{E}}^{A \rtimes G}$ are naturally equivalent because any module $M$ of $A \rtimes G$ is an $A$-module by the assignment $am:=(a,1_G)m$ and the action
$g \cdot m := (1_A,g)m$ implies the equations:
\begin{align}
    g \cdot (am) &= (1_A,g)(am)\\& = (1_A,g)(a,1_G)m\\& =(g\cdot a,g)m \\ &=(g \cdot a, 1_G)(1_A,g)m \\&=(g \cdot a)(g \cdot m) 
\end{align}
which matches the definition of graded $A$-modules endowed
with compatible $G$ actions of Eqn. \eqref{equivariant formula for action}.
We have therefore the isomorphism of K-theory groups $
\mathbf{K}^{A}_{(G,\phi)}\cong \mathbf{K}^{A \rtimes G}$, where the right-hand side is the K-theory of the obvious Banach functor.
\end{proof}

In the particular case that the automorphisms
$\tau : G \to \mathrm{MAut}_{\mathrm{gr}}(A)$ are given
by inner automorphisms $\mathrm{Ad}_{X(g)\mathbb{K}^{\phi(g)}}$ 
with $X(g) \in A$  of degree zero for all $g \in G$
as in Eqn. \eqref{inner automorpshism twisted group},
 we know that the composition rule:
 \begin{align} \label{eqn X(g)}
     X(g)  \overline{X(h)}^{\phi(g)} = \rho(g,h)X(gh)
 \end{align}
defines a 2-cocycle $\rho \in Z^2(G, \mathbb{C}^*_\phi)$
satisfying the equation: 
\begin{align} \label{2-cocyclie rho of X}
  \rho(g,hk)   \overline{\rho(h,k)}^{\phi(g)} = \rho(gh,k) \rho(g,h).  
\end{align}
This is the cocycle description of what has been presented in Prop. \ref{prop twisted reps = cocycles} with a simplified notation defined as follows:
\begin{align}
    \overline{X(h)}^{\phi(g)} := \mathbb{K}^{\phi(g)} X(h)  \ \ \ \ \mathrm{and} \ \ \ \ \overline{\rho(h,k)}^{\phi(g)}   : = \mathbb{K}^{\phi(g)}\rho(h,k).
\end{align}

Define the algebra $A_\rho[G]$ as the linear combination of elements $\lfloor a, g \rfloor $ with $a \in A$, $g \in G$ and product structure:
\begin{align} \label{product twisted algebra}
    \lfloor a, g \rfloor \lfloor b, h \rfloor = \lfloor \rho(g,h) a \overline{b}^{\phi(g)},gh \rfloor.  
\end{align}
Then we have the following isomorphism (cf. \cite[Thm. 1.8]{Karoubi-Weibel}):
\begin{proposition} \label{prop semidirect action 0 twisted}
    Suppose that the group $G$ acts on $A$ by inner automorphisms $\mathrm{Ad}_{X(g)\mathbb{K}^{\phi(g)}}$ 
with $X(g) \in A$  of degree zero for all $g \in G$
satisfying the composition rule of Eqn. \eqref{eqn X(g)}
with the 2-cocycle $\rho$ satisfying Eqn. \eqref{2-cocyclie rho of X}. Then the assignment: 
\begin{align}
    A \rtimes G & \overset{\cong}{\to} A_\rho[G]\\
    (a,g) & \mapsto \lfloor aX(g),g \rfloor
\end{align}
is an isomorphism of algebras. In particular the following isomorphism hold:
\begin{align}
    \mathbf{K}^A_{(G,\phi)} \cong \mathbf{K}^{A \rtimes G} \cong \mathbf{K}^{A_\rho[G]}.
\end{align}
\end{proposition}

\begin{proof}
    We just need to show that the assignment preserves the product structure. On the one hand we have:
    \begin{align}
        (a,g)(b,h) &= (a(g \cdot b),gh)\\
      &  = (aX(g) \overline{b}^{\phi(g)}X(g)^{-1},gh)
    \end{align}
    which is mapped to: 
    \begin{align}
         (a,g)(b,h) \mapsto & \lfloor aX(g) \overline{b}^{\phi(g)}X(g)^{-1}X(gh),gh\rfloor.
    \end{align}
    By Eqn. \eqref{eqn X(g)} we know that $X(g)^{-1}X(gh) = \rho(g,h)\overline{X(h)}^{\phi(g)}$,
    and therefore we have that:
    \begin{align}
        (a,g)(b,h) \mapsto  \lfloor\rho(g,h)aX(g) \overline{b}^{\phi(g)}\overline{X(h)}^{\phi(g)},gh\rfloor.
    \end{align}
On the other hand  we have:
\begin{align}
    \lfloor aX(g),g \rfloor\lfloor bX(h),h \rfloor = 
    \lfloor \rho(g,h) aX(g)\overline{b X(h)}^{\phi(g)},gh \rfloor.
\end{align}
The isomorphism of algebras follows. Therefore we have isomorphisms of K-theories:
\begin{align}
    \mathbf{K}^A_{(G,\phi)} \cong \mathbf{K}^{A \rtimes G} \cong \mathbf{K}^{A_\rho[G]}.
\end{align}
\end{proof}

Choosing $A=C^{p,q}_\mathbb{C}=C^{p,q}_\mathbb{R} \underset{\mathbb{R}}{\otimes} \mathbb{C}$ with complex conjugation for $G$-action, then the K-theory of the algebra $C^{p,q}_\mathbb{C}[G]$
is what the authors and Xicot\'encatl \cite[Def. 2.10]{serrano2025magneticequivariantktheory} used to define the magnetic equivariant K-theory groups:
$ \mathbf{K}^{p,q}_G. $

Denote by $\mathcal{M}^{C^{p,q}_\mathbb{C}[G]}$
the degree zero isomorphism classes of 
elements in $\mathcal{E}^{C^{p,q}_\mathbb{C}[G]}$,
and consider the restriction homomorphism
$\mathcal{M}^{C^{p,q+1}_\mathbb{C}[G]} \stackrel{\mathrm{res}}{\longrightarrow}  \mathcal{M}^{C^{p,q}_\mathbb{C}[G]} $. 
The cokernel of this map defines the 
K-theory 
\begin{align}
\mathcal{M}^{C^{p,q+1}_\mathbb{C}[G]} \stackrel{\mathrm{res}}{\longrightarrow}  \mathcal{M}^{C^{p,q}_\mathbb{C}[G]} \stackrel{\mathrm{ABS}}{\longrightarrow}
    \mathbf{K}^{p,q}_{G} \longrightarrow 0,
\end{align}
and the notation ABS
stands for the generalization of the Atiyah-Bott-Shapiro homomorphism \cite[\S 11]{abs}.

The isomorphism 
\begin{align}
    \mathbf{K}^{p,q}_G\cong \mathbf{K}^{C^{p,q}_\mathbb{C}[G]} \end{align}
between the K-theory $\mathbf{K}^{p,q}_G$ defined via the ABS homomorphism and Karoubi's K-theory $\mathbf{K}^{C^{p,q}_\mathbb{C}[G]}$ of the Banach functor functor 
\begin{align}
  \varphi^{p,q} : \mathcal{E}^{C^{p,q}_\mathbb{C}[G]} \longrightarrow \overline{\mathcal{E}}^{C^{p,q}_\mathbb{C}[G]} 
\end{align}
follows from the fact that a degree one isomorphism in $\overline{\mathcal{E}}^{C^{p,q}_\mathbb{C}[G]}$
can be lifted to a degree zero isomorphism
in $\mathcal{E}^{C^{p,q+1}_\mathbb{C}[G]}$ and vice versa.

The K-theory groups $\mathbf{K}^{p,q}_G$ for the magnetic group $(G, \phi)$ were calculated in \cite[Thm 2.5]{serrano2025magneticequivariantktheory} and are as follows:
\begin{align} \label{Coefficients K_G}
\mathbf{K}^{p,q}_G \cong (\mathrm{KO}^{p,q})^{\oplus n_\mathbb{R}} \oplus 
(\mathrm{KSp}^{p,q})^{\oplus n_\mathbb{H}} \oplus 
(\mathrm{KU}^{p,q})^{\oplus n_\mathbb{C}},
\end{align}
where $n_\mathbb{F}$ denotes the number
of irreducible magnetic representations of the group $G$ of type $\mathbb{F} \in \{\mathbb{R}, \mathbb{H}, \mathbb{C} \}$, and $\mathrm{KO}$, $\mathrm{KSp}$ and $\mathrm{KU}$ denote the K-theory of real vector bundles, of quaternionic vector bundles and of complex vector bundles, respectively.

If $\widetilde{G}$ is a magnetic $\mathbb{C}^*$ extension 
of $G$ as in Eqn. \eqref{extension of groups tilde G}, the $\widetilde{G}$-twisted magnetic $G$-equivariant K-theory groups ${}^{\widetilde{G}}\mathbf{K}^{p,q}_G$
defined in \cite[Def. 2.11]{serrano2025magneticequivariantktheory}, are precisely the K-theory groups of the twisted algebra 
$(C^{p,q}_\mathbb{C})_\rho [G]$
where $\rho$ is a 2-cocycle of the extension
as in Eqn. \eqref{2-cocyclie rho of X} (up to a canonical isomorphism). So, in this case we have:
\begin{align}
    {}^{\widetilde{G}}\mathbf{K}^{p,q}_G \cong \mathbf{K}^{(C^{p,q}_\mathbb{C})_\rho [G]},
\end{align}
where the left hand side is from 
\cite[Def. 2.11]{serrano2025magneticequivariantktheory} and the right hand side is the K-theory of the Banach functor of the twisted algebra $(C^{p,q}_\mathbb{C})_\rho [G]$ as in Eqn. \eqref{product twisted algebra}.

The twisted magnetic equivariant K-theory group $ {}^{\widetilde{G}}\mathbf{K}^{p,q}_G$ is a subgroup of the magnetic equivariant K-theory $ \mathbf{K}^{p,q}_{\widetilde{G}}$, and we can use the isomorphism of  Eqn. \eqref{Coefficients K_G} to determine them, namely:
\begin{align} \label{Coefficients twisted K_G}
{}^{\widetilde{G}}\mathbf{K}^{p,q}_G \cong (\mathrm{KO}^{p,q})^{\oplus \widetilde{n}_\mathbb{R}} \oplus 
(\mathrm{KSp}^{p,q})^{\oplus \widetilde{n}_\mathbb{H}} \oplus 
(\mathrm{KU}^{p,q})^{\oplus \widetilde{n}_\mathbb{C}},
\end{align}
where $\widetilde{n}_\mathbb{F}$ denotes the  number
of $\widetilde{G}$-twisted irreducible representations of $G$ of type $\mathbb{F} \in \{\mathbb{R}, \mathbb{H}, \mathbb{C} \}$.

\subsection{Degree shift isomorphism}
Now we can prove the degree shift isomorphism for the total K-theory $\mathbf{K}^*_{(G,\phi)}$ generalizing the 
one presented in \cite[Thm. 2.6]{serrano2025magneticequivariantktheory}. For this, we need some extra notation. 

Let $\sigma : \mathbb{Z}/2 \times \mathbb{Z}/2 \to \mathbb{C}^*$ be the 2-cocycle: 
\begin{align}
    \sigma(a,b) = e^{\pi i ab}
\end{align}
which characterizes the central extension of magnetic groups:
\begin{align}
    \mathbb{C}^* \to \mathbb{C}^*\underset{\langle (-1,2) \rangle}{\times} \mathbb{Z}/4 \to \mathbb{Z}/2, 
\end{align}
built from the central extension of magnetic groups:
\begin{align}
    \mathbb{Z}/2 \to \mathbb{Z}/4 \to \mathbb{Z}/2.
\end{align}
Here $\langle (-1,2) \rangle$ denotes the subgroup of $\mathbb{C}^* \times \mathbb{Z}/4 $ generated by $(-1,2)$.
Denote the pullback group $\widehat{G} : = \phi^* \mathbb{Z}/4$ and note that the pullback cocycle $\phi^* \sigma$ characterizes the central extension:
\begin{align}
    \mathbb{Z}/2 \to \widehat{G} \to G.
\end{align}

Now we are ready to prove the degree shift isomorphism (cf. \cite[Thm. 2.6]{serrano2025magneticequivariantktheory}):

\begin{proposition} [Degree shift isomorphism] \label{degree shift}
    Consider the $\mathbb{Z}/2$ action on $M_2(\mathbb{C})$ given by $\mathrm{Ad}_{\left(\begin{smallmatrix}
            0 & -1 \\ 1 & 0
        \end{smallmatrix}\right)\mathbb{K}}$ as in Eqn. \eqref{M2(C) generator Br(R)}. Then there is an isomorphism of algebras:
        \begin{align}
            M_2(\mathbb{C}) \rtimes \mathbb{Z}/2 \cong M_2(\mathbb{C})_{\sigma}[\mathbb{Z}/2].
        \end{align}
        If $G$ acts on $M_2(\mathbb{C})$ 
        by the pullback action of $\phi$, then:
        \begin{align}
            M_2(\mathbb{C}) \rtimes G \cong M_2(\mathbb{C})_{\phi^* \sigma}[G].
        \end{align}
        In particular, we have an induced isomorphism of K-theories:
        \begin{align}
        \mathbf{K}^{p,q+4}_{\mathbb{Z}/2} \cong {}^{\mathbb{Z}/4}\mathbf{K}^{p,q}_{\mathbb{Z}/2}, \ \ \ \ \ \ 
            \mathbf{K}^{p,q+4}_G \cong {}^{\widehat{G}}\mathbf{K}^{p,q}_G,  \ \ \ \
            \mathbf{K}^{A^{p,q+4}}_G \cong \mathbf{K}^{{A^{p,q}}_{\phi^*\sigma}[G]}, \label{degree shift isomorphism}
        \end{align}
        for $A$ any MECSGA$_{(G,\phi)}(\mathbb{C})$. 
\end{proposition}

\begin{proof}
    The isomorphisms of algebras are a simple application of Prop. \ref{prop semidirect action 0 twisted} since the cocycle induced
    by the matrix $\left(\begin{smallmatrix}
            0 & -1 \\ 1 & 0
        \end{smallmatrix}\right)$ is precisely $\sigma$.

The magnetic $\mathbb{Z}/2$-equivariant central and simple algebra $C^{0,4}_\mathbb{C}$ with automorphism given by complex conjugation $\mathbb{K}$ is similar to the complexification of the quaternion algebra. Following the notation presented in Eqns. \eqref{C04 clifford}
and \eqref{M2(C) generator Br(R)} we have the similarity of algebras:
\begin{align}
    (C^{0,4}_\mathbb{C}, \mathbb{K}) \sim \left(M_2(\mathbb{C}), \mathrm{Ad}_{\left(\begin{smallmatrix}
            0 & -1 \\ 1 & 0
        \end{smallmatrix}\right)\mathbb{K}} \right).
\end{align}
Since the algebras $M_2(\mathbb{C})_{\sigma}[\mathbb{Z}/2]$ and $\mathbb{C}_{\sigma}[\mathbb{Z}/2]$
are Morita equivalent, we have the following  Morita equivalence of algebras:
\begin{align}
    C^{0,4}_\mathbb{C} \rtimes G \sim M_2(\mathbb{C}) \rtimes G \cong M_2(\mathbb{C})_{\phi^*\sigma}[G] \sim  \mathbb{C}_{\phi^*\sigma}[G],
\end{align}
which imply the Morita equivalences:
\begin{align}
    C^{p,q+4}_\mathbb{C} \rtimes G \sim 
    (C^{p,q}_{\mathbb{C}})_{\phi^*\sigma}[G], \ \ \ \
     A^{p,q+4} \rtimes G \sim 
    {A^{p,q}}_{\phi^*\sigma}[G]
\end{align}
The isomorphism of K-theories follow.
\end{proof}

The K-theory groups $\mathbf{K}^{p,q}_{\mathbb{Z}/2}$ and ${}^{\mathbb{Z}/4}\mathbf{K}^{p,q}_{\mathbb{Z}/2}$ of trivial $\mathbb{Z}/2$-spaces are isomorphic respectively to the K-theory  $\mathrm{KO}^{p,q}$ of real vector bundles and to the symplectic K-theory $\mathrm{KSp}^*$. 
The isomorphism of Eqn. \eqref{degree shift isomorphism} realizes the well known isomorphism $\mathrm{KO}^{p,q+4} \cong \mathrm{KSp}^{p,q}$ \cite{dupontsymplectic}.

\begin{corollary}[4-periodicity] \label{corollary 4-periodicity}
Let $(G,\phi)$ be a magnetic group such that the cohomology class of $\phi^* \sigma$ is trivial, ie. $[\phi^*\sigma]=1$. Then the magnetic equivariant K-theory $\mathbf{K}_G^*$ is 4-periodic.
\end{corollary}

\begin{proof}
    If $\phi^*\sigma = \delta f$ for $f : G \to \mathbb{C}^*$, then $(\phi^*\sigma)(g,h) = f(g)\overline{f(h)}^{\phi(g)}f(gh) ^{-1}$ and the
    assignment:
    \begin{align}
        \mathbb{C}[G] &\to \mathbb{C}_{\phi^*\sigma}[G]\\
        g &\mapsto f(g)g
    \end{align}
    is an isomorphism of algebras. Therefore, the algebras
    $C^{0,4}_\mathbb{C} \rtimes G$ and $\mathbb{C}[G]$
    are Morita equivalent and therefore by Prop. \ref{degree shift} we obtain the desired 4-periodicity:
    \begin{align}
          \mathbf{K}^{A^{p,q+4}}_G \cong \mathbf{K}^{{A^{p,q}}}_G.
    \end{align}
\end{proof}

The 4-periodicity had already appeared in the calculation
of the magnetic equivariant graded Brauer group of the cyclic groups $\mathbb{Z}/2n$ for $n >1$ in Eqn. \eqref{GrBr(Z2n)}. The pullback group satisfies $\widehat{\mathbb{Z}/2n} = \mathrm{mod}_2^* (\mathbb{Z}/4 )\cong \mathbb{Z}/2n \times \mathbb{Z}/2$ and therefore $[\mathrm{mod}_2^* \sigma]=1$. Hence we have the isomorphism of groups $\mathrm{GrBr}_{(\mathbb{Z}/2n, \mathrm{mod}_2)}(\mathbb{C})\cong \mathbb{Z}/2 \times \mathbb{Z}/4$.

For the particular case of the magnetic group $(\mathbb{Z}/4, \mathrm{mod}_2)$ the total K-theory groups are:
\begin{align} \label{iso total K-theory Z4}
    \mathbf{K}^{\mathrm{tot}}_{\mathbb{Z}/4} \cong \mathbf{K}_{\mathbb{Z}/4}^* \oplus {} ^{\mathbb{Z}/8}\mathbf{K}_{\mathbb{Z}/4}^*
\end{align}
where for $0 \leq p+q < 4$ we have: 
\begin{align} \label{KZ/4 in k-theory spectra}
    \mathbf{K}_{\mathbb{Z}/4}^{p,q}\cong \mathrm{KO}^{p,q} \oplus \mathrm{KSp}^{p,q}, \ \ \ \ \  {} ^{\mathbb{Z}/8}\mathbf{K}_{\mathbb{Z}/4}^{p,q} \cong  \mathrm{KU}^{p,q}.
\end{align}
The non-trivial class in $H^2(\mathbb{Z}/4, \mathbb{C}^*_\phi) \cong \mathbb{Z}/2$ defines the 
twisted K-theory ${} ^{\mathbb{Z}/8}\mathbf{K}_{\mathbb{Z}/4}^*$, and the Clifford algebras together with the 4-periodicity induce the isomorphism of Eqn. \eqref{iso total K-theory Z4}. The isomorphism in terms of the known K-theory spectra of Eqn. \eqref{KZ/4 in k-theory spectra} follows from the calculation of the coefficients of magnetic equivariant K-theory done by the authors and Xicot\'encatl in \cite[Thm. 2.63]{serrano2025magneticequivariantktheory}.









\section*{Acknowledgments}
HS is supported by the Simons Foundation, grant SFI-MPS-T-Institutes-00007697, and the Ministry of Education and Science of the Republic of Bulgaria, through grant number DO1-239/10.12.2024. 
BU acknowledges the continuous financial support of the Max Planck Institute of Mathematics in Bonn, Germany, the International Center for Theoretical Physics in Trieste, Italy, through its Associates Program, and the Alexander Von Humboldt Foundation in Bonn, Germany.







\bibliographystyle{alpha}
\bibliography{ref}


   \end{document}